\documentclass[11pt]{amsart}

\usepackage{amscd,amssymb,amsopn,amsmath,amsthm,mathrsfs,graphics,amsfonts,enumerate,verbatim,calc,enumerate}
\usepackage[all]{xy}
\usepackage[shortlabels]{enumitem}
\usepackage{url}
\newlist{steps}{enumerate}{1}

\usepackage[OT2,OT1]{fontenc}
\newcommand\cyr{%
\renewcommand\rmdefault{wncyr}%
\renewcommand\sfdefault{wncyss}%
\renewcommand\encodingdefault{OT2}%
\normalfont
\selectfont}
\DeclareTextFontCommand{\textcyr}{\cyr} 

\usepackage{amssymb,amsmath}

\DeclareFontFamily{OT1}{rsfs}{}
\DeclareFontShape{OT1}{rsfs}{n}{it}{<-> rsfs10}{}
\DeclareMathAlphabet{\mathscr}{OT1}{rsfs}{n}{it}

\numberwithin{equation}{section}
\newtheorem{theorem}{Theorem}[section]
\newtheorem{lemma}[theorem]{Lemma}
\newtheorem{proposition}[theorem]{Proposition}
\newtheorem{corollary}[theorem]{Corollary}
\newtheorem{claim}[theorem]{Claim}

\newtheorem{Problem}{Problem}

\newtheorem{conjecture}[theorem]{Conjecture}

\theoremstyle{definition}
\newtheorem{definition}[theorem]{Definition}
\newtheorem{remark}[theorem]{Remark}

\newtheorem{example}[theorem]{Example}

\theoremstyle{remark}

\newcommand{\Spec}{\operatorname{Spec}}

\newcommand{\gr}{\operatorname{gr}}

\newcommand{\id}{\operatorname{id}}

\newcommand{\Tor}{\operatorname{Tor}}

\newcommand{\Char}{\operatorname{char}}

\newcommand{\depth}{\operatorname{depth}}

\newcommand{\Frac}{\operatorname{Frac}}

\newcommand{\pr}{\operatorname{pr}}

\newcommand{\Pic}{\operatorname{Pic}}
\newcommand{\Proj}{\operatorname{Proj}}

\newcommand{\Alb}{\operatorname{Alb}}

\newcommand{\fm}{\frak{m}}
\newcommand{\fp}{\frak{p}}
\newcommand{\fq}{\frak{q}}
\newcommand{\fa}{\frak{a}}

\newcommand{\fn}{\frak{n}}
\newcommand{\fM}{\frak{M}}

\begin{document}
\title[On local rings without small Cohen-Macaulay algebras in mixed characteristic]
{On local rings without small Cohen-Macaulay algebras in mixed characteristic}

\author[K.Shimomoto]{Kazuma Shimomoto}
\address{Department of Mathematics, Tokyo Institute of Technology, 2-12-1 Ookayama, Meguro, Tokyo 152-8551, Japan}
\email{shimomotokazuma@gmail.com}

\author[E.Tavanfar]{Ehsan Tavanfar}
\address{School of Mathematics, Institute for Research in Fundamental Sciences (IPM), P. O. Box: 19395-5746, Tehran, Iran}
\email{tavanfar@ipm.ir}

\thanks{The research of the first author was partially supported by JSPS Grant-in-Aid for Scientific Research(C) 23K03077.\\
The research of the second author was supported by a grant from IPM}



\begin{abstract} 
For any $d\ge 4$, by deformation theory of schemes, we present examples of (complete or excellent) $d$-dimensional mixed characteristic normal local domains admitting no small Cohen-Macaulay algebra, but admitting instances of small (maximal) Cohen-Macaulay modules. It is  also shown that a graded normal domain   over a field whose Proj is an Abelian variety admits a graded small (maximal) Cohen-Macaulay module.
\end{abstract}

\maketitle
\tableofcontents

\section{Introduction}

The notion of Cohen-Macaulay modules has been occupying a prominent status in the study of commutative rings, algebraic geometry, and representation theory. 
Recall that a \textit{small Cohen-Macaulay module} (or a maximal Cohen-Macaulay module in the literature) over a Noetherian local ring $(R,\fm)$ is a nonzero finitely generated $R$-module $S$ such that $\depth(S)=\dim(R)$. When $S$ is an $R$-algebra, we say that $S$ is a \textit{small Cohen-Macaulay algebra}. A (not necessarily finitely generated) module $M$ over $(R,\fm)$ is a \textit{big Cohen-Macaulay $R$-module} if some system of parameters of $R$ is a regular sequence on $M$ and $M \ne \fm M$. A big Cohen-Macaulay module is \textit{balanced} if every system of parameters of $R$ satisfies the above property. A result of Holm \cite{CMModules} connects small Cohen-Macaulay modules with big Cohen-Macaulay modules. Namely, every balanced big Cohen-Macaulay module over a  Cohen-Macaulay local ring $(R,\fm)$ is a direct limit of small Cohen-Macaulay $R$-modules. This result also reveals the abundance of Cohen-Macaulay modules over a Cohen-Macaulay local ring. For this reason, there is a huge amount of research articles on the representation theory on Cohen-Macaulay modules.

However, it is unknown whether a general complete local ring admits a small Cohen-Macaulay module; recall that a local domain admitting a small Cohen-Macaulay module has to be universally catenary.

\begin{conjecture}[Hochster, small Cohen-Macaulay conjecture (1970s)]
Every complete local ring admits a small Cohen-Macaulay module.
\end{conjecture}


Since the conjecture was proposed, there has been little progress towards either positive or negative direction. In contrast to our little knowledge about existence/non-existence of small Cohen-Macaulay modules, big Cohen-Macaulay modules (algebras) are now known to exist  in general for every local ring. In fact, for residually prime characteristic rings, the absolute integral closure does the job:

\begin{theorem}[Bhatt, Hochster-Huneke; see \cite{AbsoluteIntegral}, \cite{BigCMAlg}, and \cite{AbsoluteInt}]
\label{TheoremR+BigCM}
Let $(R,\fm)$ be an excellent local domain with residue characteristic $p>0$. Let $x_1,\ldots,x_d$ be a system of parameters of $R$ (we take $x_1=p$ in the mixed characteristic case). Then $x_1,\ldots,x_d$ is a regular sequence on the absolute integral closure of $R$.
\end{theorem}

Our main aim in this paper is to open a further (small) window to the landscape of the small Cohen-Macaulay conjecture by constructing examples (Theorem \ref{non-existence1} and Theorem \ref{SegreIso2} below) which, if not partial answers to Hochster's conjecture, they offer instances of complete mixed characteristic normal local domains without any small Cohen-Macaulay algebra, but admitting small Cohen-Macaulay modules. Particularly  in Theorem \ref{non-existence1}, we utilize the deformation theory to construct a complete local ring  in mixed characteristic satisfying our desired conclusion:

\begin{theorem}
\label{non-existence1}
Let $(V,\pi,k)$ be a complete discrete valuation ring of mixed characteristic such that $k$ is an algebraically closed field of characteristic $p>0$. Let $h:\mathcal{X} \to \Spec(V)$ be a flat surjective projective morphism with an ample line bundle $\mathcal{L}$, and the pair $(\mathcal{X},\mathcal{L})$ specializes to $(X,L)$ along the closed fiber of $h$. Define the section rings:
$$
R(\mathcal{X},\mathcal{L}):=\bigoplus_{n \ge 0}H^0(\mathcal{X},\mathcal{L}^n)~\mbox{and}~R(X,L):=\bigoplus_{n \ge 0}H^0(X,L^n).
$$
Then one can construct an example of a complete local normal domain $(R,\fm,k)$ of mixed characteristic with an algebraically closed residue field such that the following assertion holds.
\begin{enumerate}
\item[$\bullet$]
$R$ does not admit a small Cohen-Macaulay algebra, while it admits a small Cohen-Macaulay module. Moreover, 
$R$ 
is the completion  of  $R(\mathcal{X},\mathcal{L})$ 
along the graded maximal ideal $(\pi)+R(\mathcal{X},\mathcal{L})_{+}$, and such that $R(\mathcal{X},\mathcal{L})$ deforms $R(X,L)$.
\end{enumerate}
\end{theorem}

We will actually demonstrate this theorem twice in Section 3 by constructing two distinct families of examples. In Theorem \ref{TheoremK3SurfaceDeformation}, $\mathcal{X}$ is obtained by deforming the product of a K3 surface of finite height and the projective line. Here $\dim(R)=5$. In Theorem \ref{TheoremAbelianSurface}, $\mathcal{X}$ is the deformation of an Abelian surface (by setting $d=2$ in the statement). Here we can get $\dim(R)$ as small as $4$, but we must require $\text{char}(k)\neq 2$.

The same result concerning the non-existence of small Cohen-Macaulay algebras in positive characteristic has been proved by Bhatt \cite{SmallCM}; see also \cite[Example 5.3]{SannaiSinghGalois} for an explicit  graded  example in prime characteristic. The proof of Theorem \ref{non-existence1} is obtained by using Bhatt's criterion  via rigid cohomology (\cite{SmallCM}) as well as 
the graded isomorphism  $R(\mathcal{X},\mathcal{L})/(\pi) \cong R(X,L)$ which is proved in Proposition \ref{crucial-isom} under some conditions. So, our proof  is based on the  $p$-adic deformation of projective varieties. Here, Deligne's theorem on $K3$ surfaces in characteristic $p>0$ \cite{K3Surface} is applied in Theorem \ref{TheoremK3SurfaceDeformation}. Simiarly,   Mumford-Norman-Oort's theorem on Abelian varieties \cite{Norman} and \cite{NormanOort} is applied in the proof of Theorem \ref{TheoremAbelianSurface}. We also utilize various cohomology theories; $\ell$-adic \'etale cohomology, crystalline cohomology and so on. The authors hope that the method for lifting singularities used in the proof of Theorem \ref{non-existence1} will be useful in constructing interesting singularities in mixed characteristic from the singularities in positive characteristic (see Remark \ref{finalremark}). Also, in the course of the proof of Theorem \ref{non-existence1}, we derive the mentioned graded result in the abstract on the existence of maximal Cohen-Macaulay modules,  for which we refer to Corollary \ref{CorollaryAbelian} (see also Corollary \ref{CorollaryBlowUpOfCMadmitMCM}  for another 
result concerning the existence of maximal Cohen-Macaulay modules). 


The proof of our second class of examples presented in the following result is greatly inspired by \cite{PlusClosure}, where Heitmann studied various properties of the plus closure in mixed characteristic and he derived several outcomes about blow-up algebras of the type in the following result (see also Example \ref{ExamplePrimeNonExistenceExample} and Proposition \ref{PropositionCalabiYau}).

\begin{theorem}
\label{SegreIso2}
Let $V$ be an unramified discrete valuation ring of mixed characteristic $(0,p)$ with algebraically closed residue field $k$ of characteristic 
prime to $n$, where $n\ge 3$ is a natural number. Let $R$ be the blow-up algebra
$$
R=\big(V[X_{1},\ldots,X_{n}]/(X_{1}^{n}+\cdots+X_{n}^{n})\big)[X_{1}t,\ldots,X_{n}t]
$$
and let $\mathfrak{M}=(p,X_1,\ldots,X_n,X_1t,\ldots,X_nt)$. Then the completion of the localization $R_{\mathfrak{M}}$ does not admit any small Cohen-Macaulay algebra but its ideal $I:=(X_{1}t,\ldots,X_{n}t)$ is a small Cohen-Macaulay module.
\end{theorem}

The first part of Section 2 is expository. It concerns the absolute integral closure and includes the following proposition:

\begin{proposition}
\label{non-existence0}
The following assertions hold.
\begin{enumerate}
\item\label{itm:R+nCM} (cf. Theorem \ref{TheoremR+BigCM})
Let $(R,\fm)$ be a complete local domain of mixed characteristic $p>0$ of dimension  $d\ge 4$. Let $p,x_2,\ldots,x_d$ be a system of parameters of $R$ and let $R^+$ be the absolute integral closure. Then $x_2,\ldots,x_d$ does not form a regular sequence on $R^+$.

\item\label{itm:nCMFiniteAlgebra}
For any $d\ge 4$ and any $d$-dimensional complete local normal domain $(R,\fm,k)$ of mixed characteristic 
there exists a   finite  (complete) local normal domain extension  $R'$ of $R$    such that $R'$ does not admit a small Cohen-Macaulay algebra.
\end{enumerate}
\end{proposition}

In characteristic $0$, it is quite easy to obtain the same result  using the normalized trace map. The proof of Proposition \ref{non-existence0} is reduced to the failure of the balanced Cohen-Macaulay property on the absolute integral closure of a complete local domain of mixed characteristic in dimension at least $4$, which itself is deduced by reduction to equal-characteristic $0$. This result is more or less considered as folklore. However, as the authors were not able to find an actual proof  in the literature and it seems that the statement of Proposition \ref{non-existence0} has not appeared in print, we decided to give a proof. The new result of the first subsection of Section 2 is Corollary \ref{BBCMDim3}.



\section{Failure of the Cohen-Macaulay property for small algebras}
\label{SectionCM}

\subsection{Some facts concerning the absolute integral closures}
Let $R$ be an integral domain with field of fractions $K$. Let $\overline{K}$ be a fixed algebraic closure of $K$. Then the \textit{absolute integral closure} of $R$ is defined to be the integral closure of $R$ in $\overline{K}$ and we denote this ring by $R^+$.

%

It is perhaps necessary to  begin with the remark below on the statement of Proposition \ref{non-existence0} which is the first result of our paper. 

\begin{remark}
\begin{itemize}
\item
Concerning the first part of   Proposition \ref{non-existence0}, if one only wants to  discuss the existence of an instance of a  ring $R$  as in Proposition \ref{non-existence0}(\ref{itm:R+nCM}) with $x_2,\ldots,x_d$ not a regular sequence on $R^+$, then in view of Example \ref{ExamplePrimeNonExistenceExample}  this   holds by an easier observation than the proof of Proposition \ref{non-existence0}(\ref{itm:R+nCM}). However, if one wants to show  that this phenomenon holds for an arbitrary ring $R$ as in the statement, then the relatively long proof of Proposition \ref{non-existence0}(\ref{itm:R+nCM}) seems to us to be necessary.

\item 
In contrast to the equal characteristic zero case, by \cite[Example 3.2]{SridharExistence}, the statement of Proposition \ref{non-existence0}(\ref{itm:nCMFiniteAlgebra})  can not be improved by saying that $R$,  itself necessarily, does not admit a small Cohen-Macaulay algebra. 
\end{itemize}
\end{remark}



In some other articles, the following proof  is cited as  ``Proof of Proposition 2.1", which is  with respect to  the numberings of the results of the  earlier versions of our paper.
\begin{proof}[Proof of Proposition \ref{non-existence0}(1)] 
Let $(V,p,k)$ be a coefficient ring for $R$ and let
$$
A=V[[x_2,\ldots,x_d]]
$$
be the Noether normalization of $R$ with a regular system of parameters $p,x_2,\ldots,x_d$. Consequently, $\big(A,(p,x_2,\ldots,x_d)\big)$ is a complete regular local ring over which $R$ is module-finite, and $A^+=R^+$. Let $\fp:=(x_2,\ldots,x_d)A$, a height $d-1$ prime ideal of $A$.


Setting  $K:=\Frac(V)$ and
$$
A':=K[X_2,\ldots,X_d],~\fq:=(X_2,\ldots,X_d),
$$
the induced homomorphism $K \to A_{\fp}$ is extended to a homomorphism of regular local rings
$$
f:A'_{\fq}\overset{}{\longrightarrow} A_{\fp},~X_i\mapsto x_i.
$$
The local homomorphism $f$ is flat by virtue of \cite[II, Lemma 57]{AndreHomologie}, because
$$
\forall\ i\ge 1,~\Tor^{A'_{\fq}}_{i}(A'_{\fq}/\fq A'_{\fq},A_{\fp})\cong H_{i}(X_2,\ldots,X_d;A_{\fp})\cong H_{i}(x_2,\ldots,x_d;A_{\fp})=0.
$$


Next, we continue the proof by  proving the following claim:

\textbf{Claim:}  There is  a module-finite normal local domain extension $S$ of $A'_{\fq}$ which is not Cohen-Macaulay.

\textit{Proof of the claim.} We  argue as follows:

\begin{enumerate}
\item[$\bullet$]
We consider the Segre product (cf. Lemma \ref{LemmaBlowUpAndSegreProduct})
$$
B:=\big(K[a_2,\ldots,a_d]/(a_2^n+\cdots+a_d^n)\big)\ \# (K[b,c]),~\mbox{where}~n\ge d-1.
$$

\item[$\bullet$]
The ring $B$ is a standard graded $K$-algebra which has dimension $d-1$ in view of \cite[Theorem (4.2.3)(i)]{GotoWatanabeOnGraded}.

\item[$\bullet$]
By our convention that $n\ge d-1$, $K[a_2,\ldots,a_d]/(a_2^n+\cdots+a_d^n)$ has non-negative $a$-invariant $n-d+1$ by \cite[14.5.27 Exercise]{BrodmannSharpLocal}. Consequently, $B$ is not Cohen-Macaulay in the light of \cite[Theorem (4.2.3)(ii)]{GotoWatanabeOnGraded}.

\item[$\bullet$]
Since $K[b,c]$ and $K[a_2,\ldots,a_d]/(a_2^n+\cdots+a_d^n)$ are both geometrically normal $K$-algebras by \cite[Tag 0380 and Tag 037Z]{StacksProject} (because the field $K$ has characteristic zero), we can deduce that $B$ is a normal domain from \cite[Tag 06DF]{StacksProject} and \cite[Remark (4.0.3) (iv)]{GotoWatanabeOnGraded}.

\item [$\bullet$]
We consider $B$ as an $A'$-algebra by sending the variables $X_2,\ldots,X_d$ to a homogeneous system of parameters of $B$. Then $B$ is module-finite over $A'$ (see e.g. \cite[Theorem 1.5.17]{BrunsHerzogCohenMacaulay}).

\item[$\bullet$]
Finally, setting
$$
S:=B_{\fm_B},
$$
where $\fm_B$ is the unique homogeneous maximal ideal, $S$ is a normal non-Cohen-Macaulay module-finite extension domain of $A'_{\fq}$, as needed. Only the finiteness of $S$ over $A'_{\fq}$ needs some explanation: Notice that $B_{\fq}$ is a local ring with maximal ideal $\fm_BB_{\fq}$, because any maximal ideal of $B_{\fq}$ contracts to $\fq A'_{\fq}$ by \cite[Lemma 2, page 66]{Matsumura} and hence it contains a homogeneous system of parameters for $B$. Consequently, $B_{\fq}\cong B_{\fm_B}=S$ is module-finite over $A'_{\fq}$.
\end{enumerate}

Let $S$ be  as in the statement of the above claim. As the base change $f_S:S\rightarrow S\otimes_{A'_{\fq}}A_{\fp}$ is faithfully flat and $S$ has dimension $d-1$, we have $\dim(S\otimes_{A'_{\fq}}A_{\fp})\ge d-1$.  On the other hand,  the ring homomorphism $A_{\fp}\rightarrow S\otimes_{A'_{\fq}}A_{\fp}$ is module-finite, as so is $A'_{\fq}\rightarrow S$. From these two facts, we conclude that $\dim(S\otimes_{A'_{\fq}}A_{\fp})=d-1$ and $S\otimes_{A'_{\fq}}A_{\fp}$ is a  module-finite extension of $A_\fp$. Let $\fn$ be an arbitrary maximal ideal of $S\otimes_{A'_{\fq}}A_{\fp}$. Then,  since it is an integral extension of $A_{\fp}$, $\fn$ contains $(1\otimes x_2,\ldots,1\otimes x_d)$, from which we deduce that it contains $\fm_S (S\otimes_{A'_{\fq}}A_{\fp})$, where $\fm_S$ is the unique maximal ideal of $S$. But $\fm_S(S\otimes_{A'_{\fq}}A_{\fp})$ is a maximal ideal of $S\otimes_{A'_{\fq}}A_{\fp}$, because $A'_{\fq} \rightarrow A_{\fp}$ induces an isomorphism on residue fields and the maximal ideal of the target is generated by the maximal ideal of the domain. Therefore  $\fn=\fm_S(S\otimes_{A'_{\fq}}A_{\fp})$. Thus  $\big(S\otimes_{A'_{\fq}}A_{\fp},\fn\big)$ is a local ring that is a local flat extension of the non-Cohen-Macaulay ring $S$. Consequently  $\big(S\otimes_{A'_{\fq}}A_{\fp},\fn\big)$  is also a non-Cohen-Macaulay local  ring.

The completion of the base change $f_S:S\rightarrow S\otimes_{A'_{\fq}}A_{\fp}$ of $f$ is an isomorphism by \cite[Tag 0AGX]{StacksProject}, because it is a flat homomorphism inducing an isomorphism on residue fields and the maximal ideal of the target is generated by the maximal ideal of the domain. Since $S$ is an excellent  normal local domain, $\widehat{S}\cong \widehat{S\otimes_{A'_{\fq}}A_{\fp}}$ is normal, a fortiori $S\otimes_{A'_{\fq}}A_{\fp}$ is also a normal local domain. 
     
Being a module-finite domain extension of $A_{\fp}$, we get the second isomorphism in
\begin{equation}
\label{EquationAbsoluteIntegralClosure}
(A^+)_{\fp}\cong (A_{\fp})^+\cong (S\otimes_{A'_{\fq}}A_{\fp})^+.
\end{equation} 
     
Since $S\otimes_{A'_{\fq}}A_{\fp}$ is not Cohen-Macaulay,  the image of $x_2,\ldots,x_d$ is not a regular sequence on $S\otimes_{A'_{\fq}}A_{\fp}$ (it forms a system of parameters). For any module-finite extension domain $S'$ of $S\otimes_{A'_{\fq}}A_{\fp}$, it is well-known that the normalized trace map splits the inclusion $S\otimes_{A'_{\fq}}A_{\fp}\rightarrow S'$  because $S\otimes_{A'_{\fq}}A_{\fp}$ is normal. The splitting implies that any bad relation among the images of $x_2,\ldots,x_d$ in $S\otimes_{A'_{\fq}}A_{\fp}$ preventing regularity remains a (non-trivial) bad relation on any such $S'$. Hence $x_2,\ldots,x_d$ is not a regular sequence on $$(S\otimes_{A'_{\fq}}A_{\fp})^+\overset{\text{(\ref{EquationAbsoluteIntegralClosure})}}{\cong}(A^+)_{\fp}\cong (R^+)_{\fp}$$ as well. In particular it is not a regular sequence on $R^+$ as was to be proved.


(\ref{itm:nCMFiniteAlgebra})
Fix some $d\ge 4$ and some complete local normal domain $R$ as in the statement. To obtain a contradiction, assume that every finite  normal local domain extension of  $R$ has a small Cohen-Macaulay algebra. Let $R'$ denote such a finite normal (complete local) domain extension and consider the inclusion map $R' \to R^+$. Then $R'$ has a small Cohen-Macaulay algebra $T$ and let $P \subset T$ be a minimal prime such that $R' \cap P=(0)$. Then $R' \to T/P$ is a module-finite extension of complete local domains. This gives a factorization:
$$
R' \to T \twoheadrightarrow T/P \hookrightarrow R^+.
$$
Using this composite map, it follows that any relation among any (part of) a system of parameters of $R'$ becomes trivial in $R^+$. Since $R^+$ is a filtered colimit of normal local 
domains that are module-finite over $R$, the above fact shows that every (part of)  a system of parameters of $R$ is regular on $R^+$. But this provides a contradiction to part (\ref{itm:R+nCM}).
\end{proof}

Let us again recall the following remarkable result (see \cite{AbsoluteIntegral}).

\begin{theorem}[Bhatt]
\label{Bhattabsolute}
Let $(R,\fm)$ be an excellent local domain of mixed characteristic with a system of parameters $p,x_2,\ldots,x_d$. Then $x_2,\ldots,x_d$ forms a regular sequence on $R^+/p^nR^+$ for any integer $n>0$.
\end{theorem}

An almost variant of Theorem \ref{Bhattabsolute} in dimension $3$ was obtained previously by Heitmann in \cite{ExtendedPlus}. As a corollary of Theorem \ref{Bhattabsolute}, we have the following result.\footnote{By a result of Kawasaki (\cite[Corollary 1.2]{Kaw01}), it is known that any excellent local ring is a homomorphic image of a Cohen-Macaulay ring.}

\begin{corollary}\label{BBCMDim3}
(cf. Proposition \ref{non-existence0}(1)) Let $(R,\fm)$ be a $3$-dimensional excellent local domain of mixed characteristic such that $R$ is a homomorphic image of a Gorenstein local ring. Then every system of parameters of $R$ is regular on $R^+$.
\end{corollary}

\begin{proof}
By Theorem \ref{Bhattabsolute}, we have $H^i_\fm(R^+)=0$ for $0 \le i \le 2$.
The rest of the proof is just an adaptation of the proof of \cite[Corollary 2.3]{AbsoluteInt} to our situation, and the details are omitted. However, this proof does not work in mixed characteristic of dimension at least $4$ or equal characteristic zero in dimension at least $3$.
\end{proof}

In general, it is not easy to understand intrinsic properties on absolutely integrally closed domains. For instance, one can prove that every finitely generated ideal of $\mathbb{Z}^+$ is principal by using the finiteness of the class group of the ring of algebraic numbers. If $A$ is a Dedekind domain, then the localization of $A^+$ at every maximal ideal is a valuation ring of rank $1$. In dimension $2$, we have the following clean result.

\begin{proposition}
\label{GorRing}
Let $(R,\fm)$ be a complete regular local domain of dimension $2$ containing $\mathbb{Q}$. Then  the absolute integral closure of $R$ is a filtered colimit of module-finite extensions $R \to T$ such that $T$ is a normal Gorenstein local domain. 

In other words, every $2$-dimensional complete local domain containing $\mathbb{Q}$ has a module-finite extension that is normal and Gorenstein.
\end{proposition}

\begin{proof}
Start with a module-finite extension of normal local domains $R \to S$. It suffices to prove that there is a further extension $R \to S \to T$ such that $T$ is normal and Gorenstein. To this aim, we may assume that $R \to S$ is a normal extension, that is a finite extension whose induced fraction fields extension is Galois. Then the statement follows from the following two claims:

The first claim is well-known to experts, see for example \cite[Section 4]{SinghCyclic} or \cite[Theorem 2.4 and Corollary 2.5]{LocalRamification} where in the second reference a further assumption on the existence of $n$-th roots of the unity is imposed.

\textbf{Claim 1:}
\textit{Let $R$ be a complete  normal local ring containing $\mathbb{Q}$ and let $\fa$ be a pure height $1$ ideal of $R$ such that $\fa$ has finite order $n$ as an element of the divisor class group of $R$. Suppose that $\fa^{(n)}=aR$ for some $a\in R$. Then the cyclic cover
$$
R':=(R\oplus \fa t\oplus \fa^{(2)}t^2\oplus\cdots\oplus\fa^{(k)}t^k\oplus\cdots)/(at^n-1)
$$
of $R$ is a complete local normal domain such that the natural extension $R\overset{r\mapsto r}{\longrightarrow} R'$ is \'etale in codimension $1$. Moreover, $(\fa\otimes_RR')^{**}\cong R'$  where $^{**}$ denotes the double $R'$-dual.}

\textit{Proof of Claim 1:}
The normality of $R'$ follows if one shows that $R\rightarrow R'$ is \'etale in codimension $1$. This \'etale property  is explained in  \cite[Section 4]{SinghCyclic}. Finally the isomorphism $(\fa\otimes_RR')^{**}\cong R'$ also follows from the same arguments as in the proof of \cite[Corollary 2.5]{LocalRamification} (here, note that the fraction field of $R'$ is $\Frac(R)$-isomorphic to the cyclic extension $\Frac(R)[X]/(X^n-a)$ of $\Frac(R)$ and since $X \pmod{X^n-a}$ is obviously integral over $R$ and $R'$ is normal so $X \pmod{X^n-a}=\sqrt[n]{a}\in R'$).

The second claim is also about obviating the algebraically closedness condition in \cite[Corollary   2.8]{LocalRamification} on the residue field.

\textbf{Claim 2:} \textit{(cf. \cite[Corollary  2.8]{LocalRamification})
Let $R$ be a complete local domain containing $\mathbb{Q}$. Then there exists a finite extension $R \hookrightarrow R'$ such that $R'$ is quasi-Gorenstein and normal.}

\textit{Proof of  Claim 2:}
Replacing $R$ with its integral closure in the Galois closure of the field extension $[\Frac(R):\Frac(B)]$, where $B$ is a Noether normalization of $R$, we can assume that $R$ is a normal extension of the complete regular local ring $B$. Then our statement follows by the same argument as in the proof of Griffith's result \cite[Corollary 2.8]{LocalRamification}, after noticing that in view of Claim  1 one can relax the condition on the existence of the roots of the unity in \cite[Theorem 2.4]{LocalRamification}, thence the condition on the algebraically closedness  of the residue field in the statement of \cite[Corollary 2.8]{LocalRamification} can also be relaxed (we notice that to deduce the statement one should apply Claim 1  to the canonical ideal of $R$ which is of finite order as an element of the divisor class group of $R$,  as $R$ is a normal extension $B$). 
\end{proof}

\begin{remark}
Recall that a local ring $(R,\fm)$ is \textit{quasi-Gorenstein} if $R$ admits a canonical module $\omega_R$ such that $\omega_R \cong R$ as $R$-modules. A Cohen-Macaulay local ring is quasi-Gorenstein if and only if it is Gorenstein. So,  quasi-Gorensteinness and Gorensteinness are equivalent in dimension $2$. A \textit{quasi-Gorenstein normal cover} of a local domain $R$ is a module-finite quasi-Gorenstein normal local extension of $R$.
\begin{itemize}
\item
In \cite[Theorem 2.5]{TavanfarASufficient}, a (weaker) version of Proposition \ref{GorRing} is proved in any dimension and any characteristic, where Gorensteinness has to be replaced with quasi-Gorensteinness, but Serre's $(R_1)$ is replaced by being complete intersection in codimension $\le 1$ (cf.  the third bullet comment below). Then \cite[Theorem 2.5]{TavanfarASufficient}, as discussed in \cite{TavanfarASufficient}, is used to show that how this property can be useful for proving the possible almost vanishing of  $H^{d-1}_\fm(R^+)$ in equal characteristic zero\footnote{Or in mixed characteristic, but in   this case it is now known that the local cohomology vanishes \cite{AbsoluteIntegral}.} and in arbitrary dimension $d$ (if this open question is going to have an affirmative answer), provided that every part of system of parameters of length $3$ is an almost regular sequence in a certain strong sense (see \cite[Statement $\mathcal{A}$, page 2579]{TavanfarASufficient} and the point is that while the dimension $d$ is arbitrary, the condition is imposed on a part of system of parameters of length $3$).

\item
Recently, it was claimed and proved in \cite[Lemma 3.5]{Quasi-Gorenstein} that any essentially of finite type normal domain over a field of characteristic $\ne 2$ admits a quasi-Gorenstein normal cover.

\item
Fortunately, there is an under preparation paper/project in which the existence of quasi-Gorenstein normal covers is planned to be established in an arbitrary characteristic for complete local normal domains.
\end{itemize}
\end{remark}

\section{Constructing complete local domains in mixed characteristic}
\subsection{Blow-up algebras over $p$-adic rings}

We prove Theorem \ref{SegreIso2}, following a construction given by Heitmann (see also \cite[Example (4.4.13)]{GotoWatanabeOnGraded} for another construction using Segre product).

\begin{example}
\label{ExamplePrimeNonExistenceExample}
It was pointed out to the authors (with an outlined proof) that what follows has long been the prime example for the folklore fact that there is an example of a mixed characteristic normal  local domain without any small Cohen-Macaulay algebra:

Let $A$ be either a field of characteristic not equal to three or the $p$-adic integers for some prime $p\neq 3$. Let $B=A[x,y,z]/(x^3+y^3+z^3)$ and let $R$ be the blow-up of $B$ at its maximal homogeneous ideal $(x,y,z)$. Then $R$ is a normal domain which does not have a small Cohen-Macaulay algebra. But, the irrelevant ideal $(xt,yt,zt)$ of $R$ is a maximal Cohen-Macaulay $R$-module (cf. Remark \ref{RemarkFermatEllipticCurve}\ref{itm:HochsterExample} as well as Corollary \ref{CorollaryBlowUpOfCMadmitMCM}).
	
Our Theorem \ref{SegreIso2} is the higher dimensional extension of this example.  While Theorem 1.4 can be proved by an algebraic proof as was outlined to the authors, we provide a proof for a more general statement (Proposition \ref{PropositionCalabiYau}) with an algebraic geometric flavor which is a potential source for  many other different but similar examples. Moreover, we derive Corollary \ref{CorollaryBlowUpOfCMadmitMCM} as a new result which is perhaps intrinsically interesting (this might be known to some experts and of course Example \ref{ExamplePrimeNonExistenceExample} led us to this corollary). 
\end{example}

\begin{lemma}\label{LemmaBlowUpAndSegreProduct}
Let $R$ be a standard normal $\mathbb{N}_0$-graded domain that is finitely generated over a perfect field $R_{[0]}=k$.
Then $R[R_+t]$ is normal, and $R\# k[A,B]$ and $ R[R_+t]$ are isomorphic to each other by a non-graded isomorphism mapping the homogeneous maximal ideal of the source onto the homogeneous maximal ideal of the target (here $A,B$ have degree $1$ and $R[R_+t]$ is considered with its canonical grading or bigrading).
\end{lemma}

\begin{proof}
Since $k$ is  perfect,   $R\#k[A,B]$ is normal and in particular it is  a domain (\cite[Tag 0380 and Tag 06DF]{StacksProject} and \cite[Remark (4.0.3)(iv)]{GotoWatanabeOnGraded}).

Let $x_1,\ldots,x_n$ be a basis of $R_{[1]}$ over $k$. The Segre product $R\# k[A,B]$ coincides with $k[Z_{1},\ldots,Z_{2n}]/\mathfrak{a}$, where $\mathfrak{a}$ coincides with the kernel of the (graded) ring homomorphism
$$
\varphi:k[Z_{1},\ldots,Z_{2n}]\rightarrow R\#k[A,B],\ \ \begin{cases} Z_{i}\mapsto x_{i}\# A, & 1\le i\le n \\
Z_{i}\mapsto x_{i-n}\# B, & n+1\le i\le2n \end{cases}.
$$
Moreover, the blow-up algebra $R[R_+t]$  coincides with $k[Z_{1},\ldots,Z_{2n}]/\mathfrak{b}$, where $\mathfrak{b}$ coincides with the kernel of the (non-graded) ring homomorphism
$$
\psi:k[Z_{1},\ldots,Z_{2n}] \rightarrow R[x_{1}t,\ldots,x	_{n}t]=R[R_+t],\ \ \begin{cases} Z_{i}\mapsto x_{i}, & 1\le i\le n \\
Z_{i}\mapsto x_{i-n}t, & n+1\le i\le2n \end{cases}.
$$
On the other hand, there is an obvious (non-graded) ring homomorphism
\begin{center}
$\gamma:
 R\#(k[A,B])\overset{\text{inclusion}}{\hookrightarrow}R\otimes_kk[A,B]\rightarrow R[t]$
\end{center}
such that the last ring map extends the natural embedding $R\hookrightarrow R[t]$ by the assignment $A\mapsto1$ and $B\mapsto t$. Since we clearly have $\gamma\circ\varphi=\eta \circ \psi$, where $\eta$ denotes the natural inclusion
 $R[R_+t]\hookrightarrow R[t]$ we get $\mathfrak{a}\subseteq\mathfrak{b}$. Thus, there exists a surjective ring homomorphism
$$
\pi:R\#k[A,B]\twoheadrightarrow R[R_+t].
$$
After comparing the dimensions on both sides, it follows that $\pi$ is an isomorphism. So the proof
  is complete.
However, the normality of   $R[R_+t]$ (since $R$ is standard) can also  be deduced from \cite[Corollary (1.3)]{BrumattiSimisVasconcelosNormal} (see also \cite[(2.1), page 693]{TomariWatanabe} for the more general case of the Rees algebra of a (not necessarily adic) Noetherian filtration of ideals) as $\gr_{R_+}(R)=R$ is a domain in the standard case. This normality also, alternatively, follows from \cite[Theorem 4.1]{HublSwansonDiscrete}.
\end{proof}

\begin{remark}
Let $X$ be a normal projective variety over an algebraically closed field $k$ with a very  ample line bundle $L$ on it. One might be curious  if there exists any geometry behind the isomorphism mentioned in Lemma \ref{LemmaBlowUpAndSegreProduct},
 between the    blow-up ring $R(X,L)[R(X,L)_+t]$ and the affine cone of $\big(\mathbf{P}^1_X,L\boxtimes \mathcal{O}_{\mathbf{P}^1_k}(1)\big)$. As this isomorphism is not graded, the only possible existing geometrical interpretation seems to us to be stated is  as follows:

If one considers the blow-up $R(X,L)[R(X,L)_+t]$  with its natural bigraded algebra structure and also if one endows $R(X,L)\#k[A,B]$ with a bigraded structure in such a way that $R_{[d]}\# A$ has degree $(d,0)$ and $R_{[d]}\# B$ has degree $(d,1)$, then the isomorphism in Lemma \ref{LemmaBlowUpAndSegreProduct} is  a bigraded isomorphism of $\mathbb{N}_0^2$-bigraded algebras over the field $k$. Then these isomorphic bigraded algebras combined with \cite[Theorem 8.6]{AraiEchizenyaKuranoDemazure} may enable us to interpret the rings $R(X,L)[R(X,L)_+t]$ and $R(X,L)\# k[A,B]$ as the Cox rings of some (unique) ample $\mathbb{Q}$-divisors on a (unique)  projective variety.
\end{remark}

We continue by presenting  an immediate corollary to Lemma \ref{LemmaBlowUpAndSegreProduct} which is perhaps intrinsically interesting, but it might be known to some experts because we derived it from (the proof of)  Example \ref{ExamplePrimeNonExistenceExample}.

\begin{corollary}
\label{CorollaryBlowUpOfCMadmitMCM}
Let $R$ be a  Cohen-Macaulay normal standard $\mathbb{N}_0$-graded ring of dimension $\ge 2$ over a perfect field $k=R_{[0]}$.   Then   $R[R_+t]$  admits a  finitely generated module $M$ which is maximal Cohen-Macaulay with respect to the unique homogeneous maximal ideal of $R[R_+t]$.
\end{corollary}

\begin{proof}
By Lemma \ref{LemmaBlowUpAndSegreProduct},
it suffices to observe that our statement holds for the  Segre product $R\# k[A,B]$ (with $\deg(A)=\deg(B)=1$). But  in view of \cite[Theorem (4.1.5)]{GotoWatanabeOnGraded}, the $R\#k[A,B]$-module $R(s)\#k[A,B]$ is maximal Cohen-Macaulay for $s=\max\{\alpha+1,0\}$, where $\alpha$ is the $a$-invariant of $R$.
\end{proof}

In some texts in the literature, the normal projective varieties considered in the next proposition are called \textit{Calabi-Yau varieties} (without requiring 
the  Kodaira vanishing (or arithmetically Cohen-Macaulay) condition imposed in the statement of the proposition). The $K3$ surfaces and the Abelian varieties are examples of such projective varieties. We stress that the ring $R(X,L)$ in the next proposition may not be Cohen-Macaulay (for example if $X$ is an Abelian variety of dimension $\ge 2$), in contrast to our presumption on $R$ in the statement of Corollary \ref{CorollaryBlowUpOfCMadmitMCM}. The \textit{Kodaira vanishing theorem} mentioned in the statement of the next proposition, is the sheaf cohomology vanishings appearing in  \cite[Corollaire 2.8]{DeligneIllusieRelevements} (prime characteristic) and \cite[Corollaire 2.11]{DeligneIllusieRelevements} (equal characteristic zero). One can notice that in the prime characteristic case these sheaf cohomology vanishings are bounded by $i+j<\inf\{\Char(k),d\}$, so might not hold for any  $i+j<d$. If the prime characteristic smooth projective variety  $X$ satisfies these vanishings for any ample line bundle $L$ and any $i,j$ with $i,j<d(=\dim(X))$, then we say  that $X$ satisfies the \textit{strong Kodaira vanishing theorem}.

\begin{proposition}
\label{PropositionCalabiYau}
Let $X$ be a smooth projective variety of dimension $d\ge 1$ over an algebraically closed field $k$ whose canonical divisor is trivial (i.e. $K_X\cong \mathcal{O}_X$), and let  $L$ be a very ample  divisor on $X$. Suppose that at least one of the following conditions holds:
\begin{enumerate}[(i)]
	\item \label{itm:EqaulCharacteristicZeroCase} $k$ has characteristic zero (thus $X$ satisfies the Kodaira vanishing theorem, see \cite[Corollaire 2.11]{DeligneIllusieRelevements}).
	\item \label{itm:StrongKodairaVanishingCase} $k$ has prime characteristic and $X$ satisfies the strong Kodaira vanishing theorem   which is the case  where, for example, $X$ lifts to $W_2(k)$ and $\Char(k)\ge d$ (see \cite[Corollaire 2.8]{DeligneIllusieRelevements}).
	\item \label{itm:ArithmeticallyCohenMacaulayCase} $k$ has prime characteristic and $(X,L)$ is arithmetically Cohen-Macaulay, i.e. $R(X,L)$ is Cohen-Macaulay.
\end{enumerate}  Let
$$
S:=R(X,L)[R(X,L)_+t]
$$
be the blow-up of the homogeneous maximal ideal $R(X,L)_+$, and let 
$$
I:=(R(X,L)_+t)
$$
be the irrelevant ideal of $S$. Then $S$ is not Cohen-Macaulay and $I$ is a maximal Cohen-Macaulay $S$-module.

In the geometric vein, after identifying $S$ with $R(X,L)\# k[A,B]$ (Lemma \ref{LemmaBlowUpAndSegreProduct}),
the ideal $I$ is the defining ideal of the cone of $L$ over $X\times \{0\}(=X\times \{(B)\})$ as a quotient of the cone of $$\big(X\times \mathbf{P}^1_k,L\boxtimes \mathcal{O}_{\mathbf{P}^1_k}(1)\big).
$$
\end{proposition}

\begin{proof}
We bring to the reader's attention that  the graded ring $R(X,L)$ is standard, because $L$ is very ample. We begin our proof with the following claim.
	
\textbf{Claim:} The validity of any  of \ref{itm:EqaulCharacteristicZeroCase}, \ref{itm:StrongKodairaVanishingCase} or \ref{itm:ArithmeticallyCohenMacaulayCase}, according to our hypothesis, implies that    
\begin{equation}
\label{EquationCohomologyNonVanishingCases}
H^i(X,L^n)\neq 0\ \  \text{only  if}\ \   (i,n)\in\{(j,0)\}_{1\le j\le d-1}\cup(0\times \mathbb{N}_0)\cup \big(d\times(\mathbb{Z}\backslash \mathbb{N})\big).
\end{equation}

 \textit{Proof of the claim:} If $X$ satisfies the  Kodaira vanishing theorem (condition \ref{itm:EqaulCharacteristicZeroCase}), or its strong version in  prime characteristic (condition \ref{itm:StrongKodairaVanishingCase}), then (\ref{EquationCohomologyNonVanishingCases}) follows from Serre duality, the triviality of $K_X$ and the (strong) Kodaira vanishing theorem. So it remains to discuss the case where $R(X,L)$ is a Cohen-Macaulay ring. In this case, the  cohomology sheaves $H^{i}(X,L^j)=H^{i+1}_{R(X,L)_+}(R(X,L))_{[j]}$ vanish for $0<i<d$ and all $j$ (\cite[Proposition (2.2)]{WatanabeSomeRemarks}). Finally, $H^0(X,L^{-j})=0$ for each $j\in \mathbb{N}$ by \cite[III, Exercise 7.1]{HartshorneAlgebraic}, and then   again  Serre duality and the triviality of the canonical sheaf yields $H^d(X,L^j)=0$ for $j\in \mathbb{N}$. Thus (\ref{EquationCohomologyNonVanishingCases}) holds in any case.

\vspace{2mm}

Let $Y:=X\times \mathbf{P}^1_k$ and $\mathcal{L}:=\pi_1^*(L)\otimes\pi_2^*(\mathcal{O}_{\mathbf{P}^1_k}(1))$, where $\pi_1$ and $\pi_2$ are projections to $X$ and $\mathbf{P}^1_k$, respectively. Let $\mathfrak{a}$ be the principal ideal generated by $\alpha_0A+\alpha_1B$ in $k[A,B]$ for some point  $(\alpha_0:\alpha_1)$ in the projective line, and let $\mathcal{I}$ be the ideal sheaf 
\begin{equation}
\label{EqutionSheafIdealI}
\mathcal{I}:=\widetilde{\mathfrak{a}} \mathcal{O}_{Y}\cong\pi_2^*\big(\widetilde{\mathfrak{a}}\big)\cong\pi_2^*\big(\mathcal{O}_{\mathbf{P}^1_k}(-1)\big)
\end{equation}  
of $Y$. It is readily seen that $$\Gamma_*\big(Y,\mathcal{I},\mathcal{L}\big)=\bigoplus_{n\in\mathbb{Z}}H^0(Y,\mathcal{I}\otimes\mathcal{L}^n\big)$$ is maximal Cohen-Macaulay as a module over $R\big(Y,\mathcal{L}\big)=\bigoplus_{n\in \mathbb{N}_0}H^0(Y,\mathcal{L}^n)$. Namely by \cite[Tag 0BED]{StacksProject},
\begin{align}
\label{EquationVanishingOfCohomologiesOfTheConeOfTheIDealShaef}
\forall\ (i,n)\in ([1,d]\cap \mathbb{N})\times\mathbb{Z},\ \ H^i\big(Y,\mathcal{I}\otimes \mathcal{L}^n\big)
&=H^i\big(Y,\mathcal{I}\otimes \pi^*_1(L^n)\otimes \pi_2^*(\mathcal{O}_{\mathbf{P}^1_k}(n))\big)
&\nonumber \\&\cong 
H^i\big(Y,\pi^*_1(L^n)\otimes \pi_2^*(\mathcal{O}_{\mathbf{P}^1_k}(n-1))\big)
&\nonumber \\& =
\big(H^i(X,L^n)\otimes H^0(\mathbf{P}^1_k,\mathcal{O}_{\mathbf{P}^1_k}(n-1))\big) 
&\nonumber \\&
\oplus \big( H^{i-1}(X,L^n)\otimes H^1(\mathbf{P}^1_k,\mathcal{O}_{\mathbf{P}^1_k}(n-1))\big).
\end{align}
Since $H^1\big(\mathbf{P}^1_k,\mathcal{O}_{\mathbf{P}^1_k}(m)\big)=0$ for all $m\ge -1$ and  $H^0\big(\mathbf{P}^1_k,\mathcal{O}_{\mathbf{P}^1_k}(m)\big)=0$ for all $m\le -1$, we see that the first summand vanishes when $n\le 0$ and the second summand vanishes when $n\ge 0$. In view of  display (\ref{EquationCohomologyNonVanishingCases}), both summands always vanish and so $H^i(Y,\mathcal{I}\otimes \mathcal{L}^n)=0$ for any  $(i,n)\in ([1,d]\cap \mathbb{N})\times\mathbb{Z}$. 
 From this and \cite[Proposition 2.1.5]{GrothendieckEGAIII} (see \cite[Remark 2.6(iii)]{ShimomotoTavanfarRemarks} for an English text where this statement is mentioned without a proof), we conclude that $\Gamma_*\big(Y,\mathcal{I},\mathcal{L}\big)$ is a maximal Cohen-Macaulay $R(Y,\mathcal{L})$-module.

 Thus, in view of Lemma \ref{LemmaBlowUpAndSegreProduct}
 and its proof, 
$$
R(Y,\mathcal{L})=R(X,L)\#k[A,B]\cong R(X,L)[R(X,L)_+t]=S
$$
under an isomorphism which maps the ideal  $J:=R(X,L)_+\# (Bk[A,B])$ of $R(X,L)\# k[A,B]$ to the ideal $I$ as defined in the statement. In order to complete the proof, it suffices for $\fa=Bk[A,B]$ (here we  specialize  $\alpha_0:=0$ in the above definition of $\mathfrak{a}$) to verify that  
\begin{align}
\label{EquationConeOfI}
\Gamma_*(Y,\mathcal{I},\mathcal{L})
&=
\bigoplus_{n\in\mathbb{Z}}\Big(H^0(X,L^n)\otimes H^0\big(\mathbf{P}^1_k,\widetilde{\mathfrak{a}}\otimes \mathcal{O}_{\mathbf{P}^1_k}(n)\big)\Big)
&\nonumber\\& \cong 
\bigoplus_{n\in\mathbb{N}}\Big(H^0(X,L^n)\otimes H^0\big(\mathbf{P}^1_k,\widetilde{\mathfrak{a}(n)}\big)\Big)
&\nonumber\\&
=J && \text{(\cite[(5.1.6)]{GotoWatanabeOnGraded})}.
\end{align}
We notice that, since $K_X\cong \mathcal{O}_X$, it follows from Serre duality that $H^d(X,\mathcal{O}_X)\neq 0$ from which we can deduce that $R(X,L)\#k[A,B]$  is not Cohen-Macaulay (\cite[Theorem (4.1.5)]{GotoWatanabeOnGraded}) and so   $S$ is not Cohen-Macaulay.

We also notice that, from the above display as well as display (\ref{EqutionSheafIdealI}) and \cite[Tag 01JU(1)]{StacksProject},  the ideal $J$ has the geometric description mentioned in the last sentence of the statement. More precisely, in view of \cite[Tag 01JU(1)]{StacksProject}, $X\times_k \{(B)\}=(X\times_k \mathbf{P}^1_k)\times_{\mathbf{P}^1_k} \{(B)\}$ is the closed subscheme of $Y=X\times_k\mathbf{P}^1_k$ with respect to the ideal sheaf $\mathcal{I}=\widetilde{(B)}\mathcal{O}_Y$ and thus the cone of $L$ over $X\times_k \{(B)\}$ is the same as the cone of $\mathcal{L}/\widetilde{(B)}\mathcal{L}=\mathcal{L}/\widetilde{\mathfrak{a}}\mathcal{L}$ over $Y$, i.e. $\bigoplus_{n\in\mathbb{N}_0}H^0(Y,\mathcal{L}^n/\widetilde{\mathfrak{a}}\mathcal{L}^n)$ (see \cite[Tag 04CJ(1)]{StacksProject} and \cite[III, Lemma 2.10]{HartshorneAlgebraic}). Then the exact sequences $$0\rightarrow \widetilde{\mathfrak{a}}\mathcal{L}^n\rightarrow \mathcal{L}^n\rightarrow \mathcal{L}^n/\widetilde{\mathfrak{a}}\mathcal{L}^n\rightarrow 0$$ of sheaves of $\mathcal{O}_Y$-modules and the fact that $$\forall\ n\in\mathbb{Z},\ \ H^1(Y,\widetilde{\mathfrak{a}}\mathcal{L}^n)\cong H^1(Y,\pi_2^*(\mathfrak{\widetilde{a})}\otimes \mathcal{L}^n)\overset{\text{display (\ref{EquationVanishingOfCohomologiesOfTheConeOfTheIDealShaef})}}{=}0$$ prove our claim that $J=\Gamma_*(Y,\mathcal{I},\mathcal{L})\cong \bigoplus_{n\in\mathbb{N}_0} H^0(Y,\widetilde{\mathfrak{a}}\mathcal{L}^n)$ is the defining ideal of the cone as stated in the statement.
\end{proof}

Now we give a proof of Theorem \ref{SegreIso2}.

\begin{proof}[Proof of Theorem \ref{SegreIso2}]
Let us show that $I$ is a small Cohen-Macaulay $R$-module. The ideal $I$ is a small Cohen-Macaulay module if and only if $I/pI$ is Cohen-Macaulay.
 Consider the exact sequence
$$
0\rightarrow (pR\cap I)/pI\rightarrow I/pI\rightarrow I(R/pR)\rightarrow 0.
$$
 Because 
$
(pR\cap I)/pI\cong \Tor^R_1(R/pR,R/I)\cong 0:_{V[X_1,\ldots,X_n]/(X_1^n+\cdots+X_n^n)}p=0,
$
$I/pI\cong I(R/pR).$
Since $I(R/pR)$ is equal to the ideal $(X_{1}t,\ldots,X_{n}t)$ in $$S:=\big(k[X_{1},\ldots,X_{n}]/(X_{1}^{n}+\cdots+X_{n}^{n})\big)[X_{1}t,\ldots,X_{n}t],$$ we are reduced to considering the case of blow-up algebras over a field $k$.

By Lemma \ref{LemmaBlowUpAndSegreProduct},
 $S$ is isomorphic to the Segre product
$$
\big(k[X_{1},\ldots,X_{n}]/(X_{1}^{n}+\cdots+X_{n}^{n})\big)\# k[A,B],
$$
where $k[X_{1},\ldots,X_{n}]/(X_{1}^{n}+\cdots+X_{n}^{n})$ is the cone of $\big(X,\mathcal{O}_{X}(1)\big)$ with $$X:=\Proj\big(k[X_{1},\ldots,X_{n}]/(X_{1}^{n}+\cdots+X_{n}^{n})\big)$$ and $\mathcal{O}_{X}(1)$  is the associated coherent sheaf to $\big(k[X_{1},\ldots,X_{n}]/(X_{1}^{n}+\cdots+X_{n}^{n})\big)(1)$ on $X$. Since $k[X_{1},\ldots,X_{n}]/(X_{1}^{n}+\cdots+X_{n}^{n})$ (say $R'$ from now on) is Gorenstein (a hypersurface) of $a$-invariant zero, the canonical divisor of $X$ is trivial, i.e. $$K_X(=\widetilde{\omega_{R'}}\overset{\text{deg.  0 graded iso.}}{\cong} \widetilde{R'})\cong \mathcal{O}_X.$$  
 Furthermore, $X$ is smooth because $R'$ is a standard graded ring that is an isolated singularity in view  of the Jacobian criterion (\cite[Theorem 4.4.9]{HunekeSwansonIntegral}) and our assumption on the characteristic. Consequently,  we are in the situation of Proposition \ref{PropositionCalabiYau} where the condition \ref{itm:ArithmeticallyCohenMacaulayCase} of the statement of Proposition \ref{PropositionCalabiYau} obviously holds in our current situation (as $R(X,O_X(1))=k[X_{1},\ldots,X_{n}]/(X_{1}^{n}+\cdots+X_{n}^{n})$), 
 from which we conclude that the ideal $(X_1t,\ldots,X_nt)$ is a maximal Cohen-Macaulay module over $S$ as desired (it is easily seen that, in the case where $\Char(k)\ge n-2$, we are also in the situation of Proposition 3.5\ref{itm:StrongKodairaVanishingCase} and thus we can get the same conclusion from the strong Kodaira vanishing theorem). 

It remains to discuss the non-existence of small Cohen-Macaulay algebras as mentioned in the statement. The localization of $$\big(\Frac(V)[X_1,\ldots,X_n]/(X_1^n+\cdots+X_n^n)\big)[X_1t,\dots,X_nt]$$ at the maximal ideal $(X_1,\ldots,X_n,X_1t,\ldots,X_nt)$ coincides with a localization of $(R_{\fM})_{p}$, which is not Cohen-Macaulay. Thus $(\widehat{R_{\fM}})_p$ is also not Cohen-Macaulay in view of the faithful flatness of $(R_{\fM})_p\rightarrow (\widehat{R_{\fM}})_p$. Since $R$ is normal and excellent, $(\widehat{R_{\fM}})_p$ is normal. As it has equal characteristic zero, it can not have a small Cohen-Macaulay algebra. However, any small Cohen-Macaulay algebra for $\widehat{R_{\fM}}$ would localize to a small Cohen-Macaulay algebra for $(\widehat{R_{\fM}})_p$ and the proof is complete.
\end{proof}
	

\begin{remark}\label{RemarkFermatEllipticCurve}
		 In the situation of Theorem \ref{SegreIso2},  suppose furthermore that   $n=3$.  So $R/pR$  is the blow-up  at the homogeneous maximal ideal of the cone over the Fermat elliptic curve $$\Proj\big(k[X_1,X_2,X_3]/(X_1^3+X_2^3+X_3^3)\big).$$ In this more special case, we provide the following two complimentary comments:
		 \begin{enumerate}[(i)]
		 	\item  It follows from \cite[Example 5.3]{SannaiSinghGalois} that $R/pR$ does not admit a graded small Cohen-Macaulay algebra. In particular, in this case, we can conclude that  $R$ does not admit a graded small Cohen-Macaulay algebra by reducing the statement to the prime characteristic instead of the reduction to the equal characteristic zero  (see also Remark \ref{AlbaneseVar}\ref{itm:AlbaneseVarNonExistenceOfSMC} for some other related arguments).
		 	
		 	\item\label{itm:HochsterExample} In \cite[Example 5.9]{HochsterCohenMacaulay}, the existence of maximal Cohen-Macaulay ideals has been established over a class of equal characteristic zero $3$-dimensional rings which includes our discussed blow-up  $\big(k[X_1,X_2,X_3]/(X_1^3+X_2^3+X_3^3)\big)[X_1t,X_2t,X_3t]$  (in view of Lemma \ref{LemmaBlowUpAndSegreProduct}
		 	 as well).
		 \end{enumerate}
\end{remark}

\subsection{An example arising from $p$-adic deformations of $K3$ surfaces}

We begin with a key result; see \cite[Theorem 1.3]{SmallCM}.

\begin{theorem}[Bhatt]
\label{non-existence}
Let $(X,L)$ be a normal  polarized variety over a perfect field $k$ of characteristic $p>0$ with section ring $R(X,L):=\bigoplus_{n \ge 0} H^0(X,L^n)$. Let $R$ denote the completion of $R(X,L)$ along the maximal ideal $R(X,L)_{+}$. Assume that 
$$
H^i_{\rm{rig}}(X/W(k)[\frac{1}{p}])_{<1} \ne 0~\mbox{for some}~0<i<\dim X.
$$
Then $R$ does not admit a small Cohen-Macaulay algebra.
\end{theorem}

A \textit{polarized variety} $(X,L)$ is a pair of a projective variety $X$ with some ample line bundle $L$ on it. The module $H^*_{\rm{rig}}(X/W(k)[\frac{1}{p}])$ is the rigid cohomology and $H^*_{\rm{rig}}(X/W(k)[\frac{1}{p}])_{<1}$ is the maximal eigensubspace of $H^*_{\rm{rig}}(X/W(k)[\frac{1}{p}])$ on which the Frobenius acts with slope $<1$ (``the slope of the Frobenius action" is defined by the weight decomposition, as the rigid cohomology carries a structure as an $F$-isocrystal). We recall that for smooth proper varieties over a field (as is the case in our present paper) the rigid cohomology coincides with the crystalline cohomology after tensoring with the field of fractions of $W(k)$, which has a simpler formalism. For (general) references to the crystalline cohomology we refer to \cite{BerthelotCohomologie} and \cite[Tag 07GI]{StacksProject}, and to the rigid cohomology we refer to \cite{StumRigid} and \cite{WittVectorCoh}. 

For the sake of the reader's convenience we refer to some related definitions:

Let $k$ be a perfect field  of characteristic $p>0$,  $K$ be the fraction field of $W(k)$ which is equipped with the Frobenius homomorphism $\sigma$ lifting the Frobenius homomorphism $F$ of $k$, and let $K\{f\}$ be the skew polynomial ring over $K$ whose multiplication is subject to the rule $f\kappa=\sigma(\kappa)f$. As a reference to the definition of an  \textit{$F$-isocrystal} on $k$ (resp. a \textit{convergent $F$-isocrystal} on $k$),  see \cite[Definition 3.1]{LiedtkeLectures}  (resp. set $X=k$ and $P=W(k)$ in \cite[Definition 2.1]{KedlayaNotes}) and cf. \cite[Remark 2.10]{KedlayaNotes}.  In view of \cite[Theorem 3.4 (Dieudonn\'e-Manin)]{LiedtkeLectures} (cf. \cite[Theorem 3.2]{KedlayaNotes}), when $k$ is an algebraically closed field, any $F$-isocrystal $\mathcal{E}$ on $k$ admits a  decomposition $\mathcal{E}=\bigoplus\limits_{r/s\in \mathbb{Q}_{\ge 0}}\mathcal{E}_{r/s}$  such that each $\mathcal{E}_{r/s}$ is isomorphic to a finite direct sum of  copies of the $F$-isocrystal $\mathcal{F}_{r/s}:=K\{f\}/(f^s-p^r)$ (see  \cite[Definition 3.1]{KedlayaNotes} for another description of $\mathcal{F}_{r/s}$). Each rigid cohomology  $H^*_{\rm{rig}}(X/W(k)[\frac{1}{p}])$ is a  (convergent) $F$-isocrystal on $k$ and to define its slope $<n$ part, by base change to the algebraic closure   if necessary, we can assume that $k$ is algebraically closed. Then the slope $<n$ part of $H^i_{\rm{rig}}(X/W(k)[\frac{1}{p}])$ is defined as the direct summand part $\bigoplus\limits_{r/s<n}\mathcal{E}_{r/s}$ of  $H^i_{\rm{rig}}(X/W(k)[\frac{1}{p}])$ in a decomposition as mentioned above (see \cite[Definition 3.5]{LiedtkeLectures}, \cite[Definition 3.3]{KedlayaNotes} and \cite[5.2 Proof of Theorem 1.1, page 383]{WittVectorCoh}). For the relation of  slope and  $p$-adic valuation of   eigenvalues of the Frobenius actions, one useful reference is \cite[Theorem 2.2]{ManinTheTheory}.  In particular, when $X$ comes from a  finite field by base field change, due to  Katz-Messing result \cite{KatzMessing}, there is a relation between the $p$-adic valuation of  the eigenvalues of \'etale cohomologies $H^i_{\text{et}}(X,\mathbb{Q}_\ell)$ (with $\ell\neq p$) and the slopes of $H^i_{\rm{rig}}(X/W(k)[\frac{1}{p}])$. 

Let us recall the definition of $K3$ surfaces and their heights.

\begin{definition}
A smooth projective surface $X$ over a field is called a \textit{K3 surface} if
 $K_X \cong \mathcal{O}_X$ and $H^1(X,\mathcal{O}_X)=0$. The height of a $K3$ surface $X$ is defined as the height of the formal Brauer group of $X$ (see \cite[18 Brauer Group, Definition 3.3]{Huybrechts} where the definition given therein requires for the ground field to be separably closed).
\end{definition}

\begin{lemma}
\label{non-existence2}
Let $k$ be an algebraically closed field of characteristic $p>0$. Let $Y$ be a $K3$ surface over $k$ of finite height ($i.e.$ not supersingular) and let $L_1$ be an ample line bundle on $Y$. Consider the polarized variety $$\Big(X:=Y\times \mathbf{P}_k^1,\ L:=\pi_1^*(L_1)\otimes\pi_2^*\big(\mathcal{O}_{\mathbf{P}^1_k}(1)\big)\Big)$$ where $\pi_1$ and $\pi_2$ are projections to $Y$ and $\mathbf{P}^1_k$, respectively. Then 
\begin{enumerate}[(i)]
\item
\label{itm:MCADoesNotExist}
$R(X,L)$  satisfies the conclusion of Theorem \ref{non-existence}. Thus the completion of  $R(X,L)$ does not admit a small Cohen-Macaulay algebra.

\item
\label{itm:MCMExists} $R(X,L)$ (and so its completion)  does admit a small Cohen-Macaulay module.
\end{enumerate}
\end{lemma}

\begin{proof}
Part \ref{itm:MCADoesNotExist}  is stated in \cite[Example 3.11]{SmallCM} without proof (the non-Cohen-Macaulayness of $R(X,L)$ itself will be shown in the sequel).  


\ref{itm:MCMExists} First, we recall the proof of the following fact:

\textbf{Fact:} Let $Y$ be a $K3$ surface over an algebraically closed field $k$ and let $D$ be an ample divisor on $Y$. Then $Y$ satisfies the Kodaira vanishing theorem (even in positive characteristic) and the section ring $R(Y,D)$ is Cohen-Macaulay.

\begin{proof}[Proof of the fact]
By \cite{XieACharacterization}, $Y$ satisfies the Kodaira vanishing theorem  even when $k$ has positive characteristic. By the same arguments as in the proof of \cite[Corollary 3.4(2)(a)]{KollarSingularities}, $R(Y,D)$ is Cohen-Macaulay when $k$ has characteristic $0$. But the proof of \cite[Corollary 3.4(2)(a)]{KollarSingularities}  works also for the case where $k$ has prime characteristic because of the validity of the Kodaira vanishing theorem for $Y$ in any characteristic.
\end{proof}

Let $\mathfrak{m}:=R(Y,L_1)_{+}$ be the irrelevant maximal ideal. In view of the above fact, we have
\begin{equation}
\label{EquationIsCM}
\forall\ i\neq 3,~H^i_{\fm}\big(R(Y,L_1)\big)=0.
\end{equation}
From $K_Y\cong \mathcal{O}_Y$ and Serre duality, we get the isomorphism appearing in 
\begin{equation}
\label{EquationNonVanishongInZeroDegree}
H^3_{\mathfrak{m}}\big(R(Y,L_1)\big)_{[0]}=H^2(Y,\mathcal{O}_Y)\cong H^0(Y,\mathcal{O}_Y)=k\neq 0.
\end{equation}
Also from $K_Y \cong \mathcal{O}_Y$ and the Kodaira vanishing theorem, we get
\begin{equation}
\label{EquationVanishingInPositiveDegrees}
\forall\ m>0,~H^3_{\mathfrak{m}}\big(R(Y,L_1)\big)_{[m]}=H^2(Y,L_1^m)=0.
\end{equation}
Since $k[x,y]=R\big(\mathbf{P}^1_k,\mathcal{O}_{\mathbf{P}^1_k}(1)\big)$ and pullback commutes with tensor product, it follows from \cite[Tag 0BED]{StacksProject} or \cite[Theorem 14]{KempfCohomology} for the $0$-th cohomologies that
$$
(A,\mathfrak{n}):=R(X,L) \cong R(Y,L_1)\#k[x,y].
$$ 
Then applying the formula for the local cohomologies of Segre products (see \cite[Theorem (4.1.5)]{GotoWatanabeOnGraded}) in view of (\ref{EquationIsCM}), (\ref{EquationNonVanishongInZeroDegree}) and (\ref{EquationVanishingInPositiveDegrees}), we observe that $A$ is a generalized Cohen-Macaulay graded ring of dimension $4$ whose only non-zero non-top local cohomology is $H^3_{\mathfrak{n}}(A)=H^3_{\mathfrak{n}}(A)_{[0]}=k$. Then, again the formula for local cohomology of Segre products  of modules shows that $R(Y,L_1)(1)\# k[x,y]$ is a small Cohen-Macaulay $A$-module.
\end{proof}

To prove Theorem \ref{non-existence1}, we are first going to construct a  smooth projective morphism $\mathcal{X} \to \Spec (V)$ with a polarization $\mathcal{L}$ on $\mathcal{X}$, where $W(k) \to V$ is a finite extension of discrete valuation rings. Then we will show that the completion of the section ring associated to $(\mathcal{X},\mathcal{L})$ gives what we want for Theorem \ref{non-existence1}.

\begin{proposition}
\label{crucial-isom}
Let the notation and hypotheses be as in Theorem \ref{non-existence1}, except that we now let $k$  be an arbitrary field of characteristic $p>0$ (not necessarily algebraically closed). Assume further that $H^1(X,L^n)=0$ for all $n \ge 1$ and $X$ is geometrically integral over $k$. Then there is an isomorphism of Noetherian graded rings
\begin{equation}
\label{gradedisom}
R(\mathcal{X},\mathcal{L})/(\pi) \cong R(X,L).
\end{equation}
\end{proposition}

\begin{proof}
We first prove the following claim.
\begin{claim}
The graded ring $R(\mathcal{X},\mathcal{L})$ is Noetherian.
\end{claim}

\begin{proof}[Proof of claim]
First, we prove the claim under the assumption that $\mathcal{L}$ is very ample. Then there is a closed immersion $i:\mathcal{X} \hookrightarrow \mathbf{P}^n_V$ so that $\mathcal{L}=i^*\mathcal{O}_{\mathbf{P}^n_V}(1)$, together with an exact sequence: $0 \to \mathcal{I}_\mathcal{X} \to \mathcal{O}_{\mathbf{P}^n_V} \to \mathcal{O}_{\mathcal{X}} \to 0$. Then for sufficiently large fixed $\ell>0$, we get $H^1(\mathbf{P}^n_V,\mathcal{I}_{\mathcal{X}} \otimes \mathcal{O}_{\mathbf{P}^n_V}(1)^{\otimes \ell m})=0$ for all $m \ge 1$ by Serre vanishing theorem. Then there is a surjection $R(\mathbf{P}^n_V,\mathcal{O}_{\mathbf{P}^n_V}(\ell))_{+} \to R(\mathcal{X},\mathcal{L}^{\otimes \ell})_{+}$. Since it is well known that $R(\mathbf{P}^n_V,\mathcal{O}_{\mathbf{P}^n_V}(\ell))$ and $H^0(\mathcal{X},\mathcal{O}_\mathcal{X})$
are Noetherian, it follows that $R(\mathcal{X},\mathcal{L}^{\otimes \ell})$ is Noetherian (see \cite[III, Theorem 5.1]{HartshorneAlgebraic}, and for the second finiteness see  \cite[Tag 02O5]{StacksProject}, \cite[III, Proposition 8.5]{HartshorneAlgebraic} and \cite[II, Theorem 4.9]{HartshorneAlgebraic}). The general case follows from \cite[Proposition 1.1.2.5]{CoxRing} by taking some Veronese subring of $R(\mathcal{X},\mathcal{L})$, as desired.
\end{proof}

Now we are prepared to prove the display $(\ref{gradedisom})$  in the statement. First, taking the long exact sequence associated to the short exact sequence 
$0 \to 
\widetilde{(\pi)}\mathcal{O}_{\mathcal{X}}\to \mathcal{O}_{\mathcal{X}} \to \mathcal{O}_X \to 0$, we get an injection $j_1:H^0(\mathcal{X},\mathcal{O}_{\mathcal{X}})/(\pi) \hookrightarrow H^0(X,\mathcal{O}_X)$. Since $X$ is geometrically integral over $k$, we have $H^0(X,\mathcal{O}_X)=k$. On the other hand, there is an injection $V \hookrightarrow H^0(\mathcal{X},\mathcal{O}_{\mathcal{X}})$, which induces $j_2:k \to H^0(\mathcal{X},\mathcal{O}_{\mathcal{X}})/(\pi)$. Now $j_1 \circ j_2$ is the identity, so it follows that $j_1$ is a bijection. Second, our hypothesis gives for all $n\ge 1$,  $H^1\big(\mathcal{X} \times \Spec(k),\mathcal{L}^n \otimes_V k\big)=H^1(X,L^n)=0$,
which implies that, for all $n\ge 1$, $\mathcal{L}^n$
is \textit{cohomologically flat in degree 0} along the smooth projective (flat) morphism $h:\mathcal{X} \to \Spec(V)$. By \cite[Remark 8.3.11.2]{FormalGeometry} (see also \cite[III, Theorem 12.11(a)]{HartshorneAlgebraic}) combined with the bijectivity of $j_1$, we obtain the required deformation identity appearing in the display $(\ref{gradedisom})$.
\end{proof}

\begin{theorem}\label{TheoremK3SurfaceDeformation}
	 Following the hypothesis and notation of Lemma \ref{non-existence2}, there is some complete discrete valuation ring  $(V,\pi,k)$  with  residue field $k$  and  a projective flat $\Spec(V)$-scheme  $\mathcal{Y}$  with an ample line bundle $\mathcal{L}_1$  obtained by deforming the pair $(Y,L_1)$  (thus $(\mathcal{Y},\mathcal{L}_1)$ specializes to $(Y,L_1)$ along the closed point of $V$).  Let $$(\mathcal{X},\mathcal{L}_t):=\big(\mathcal{Y}\times \mathbf{P}^1_V,\mathcal{L}_1^t\boxtimes \mathcal{O}_{\mathbf{P}^1_V}(1)\big)$$ which lifts  the pair $(X,L_t):=\big(Y\times \mathbf{P}_k^1,L_1^t\boxtimes \mathcal{O}_{\mathbf{P}^1_k}(1)\big)$.
	 
		  Then, for sufficiently large $t\in \mathbb{N}$,  the adic completion of $R(\mathcal{X},\mathcal{L}_t)$ along $(\pi)+R(\mathcal{X},\mathcal{L}_t)_+$  does not admit a small Cohen-Macaulay algebra  but it admits a small Cohen-Macaulay module (here, moreover, $R(\mathcal{X},\mathcal{L}_t)$ deforms $R(X,L_t)$).
\end{theorem} 
\begin{proof}

Over  $X=Y \times \mathbf{P}^1_k$, we let $\pi_1:X \to Y$ and $\pi_2:X \to \mathbf{P}^1_k$ be the projection maps, as in the statement of Lemma \ref{non-existence2}. Let also $L_2:=\mathcal{O}_{\mathbf{P}^1_k}(1)$, by which we consider the ample line bundle $L=\pi_1^*(L_1) \otimes \pi_2^*(L_2)$ on $X$.

By Deligne's result on lifting polarized $K3$ surfaces (see for example \cite[Theorem 8.5.27]{FormalGeometry}), there is a (possibly ramified) valuation ring $(V,\pi)$ that is module-finite over $W(k)$, together with a  smooth projective morphism $f:\mathcal{Y} \to \Spec(V)$ and an ample line bundle $\mathcal{L}_1$ on $\mathcal{Y}$ such that 
\begin{equation}
	\label{Deligne1}
	Y \cong \mathcal{Y} \times \Spec(k)~\mbox{and}~L_1 \cong \mathcal{L}_1 \otimes_V k.
\end{equation}
We apply the same procedure to $\mathbf{P}^1_k$ to get the smooth morphism $g:\mathbf{P}^1_{V} \to \Spec(V)$ with $\mathcal{L}_2$ which lifts $L_2$ (namely,
$\mathcal{L}_2:=\mathcal{O}_{\mathbf{P}^1_V}(1)$). Notice that $V/(\pi) \cong k$. Using $\mathcal{L}_1$ and $\mathcal{L}_2$, we can take, as in Lemma \ref{non-existence2}, the ample line bundle $\mathcal{L}$ on $h:\mathcal{X}=\mathcal{Y} \times_{\Spec(V)} \mathbf{P}^1_V \to \Spec(V)$ defined by $f$ and $g$: In other words,
using the projections $\eta_1:\mathcal{X} \to \mathcal{Y}$ and $\eta_2:\mathcal{X} \to \mathbf{P}^1_V$ we get the ample line bundle 	$\mathcal{L}=\eta_1^*(\mathcal{L}_1) \otimes \eta_2^*(\mathcal{L}_2)$ such that
\begin{equation}
	\label{Deligne2}
    \mathcal{L} \otimes_V k \cong L.
\end{equation}

After replacing $L_1$ (respectively, $\mathcal{L}_1$) with a sufficiently higher power, say $L_1^t$ (respectively $\mathcal{L}_1^t$)  (as in the statement), we may assume that
\begin{equation}
\label{SerreVanishing}
H^1(Y,L_1^n)=0;~\forall n \ge 0.
\end{equation}

For $n>0$, this is a consequence of Serre vanishing theorem for ample line bundles (\cite[III, Proposition 5.3]{HartshorneAlgebraic}), while the case $n=0$ is by the definition of $K3$ surfaces.  So for the rest of the proof we assume that the number $t$ in the statement is $1$ (i.e. $L_1^t=L_1,\mathcal{L}_1^t=\mathcal{L}_1,L_t=L$ and $\mathcal{L}_t=\mathcal{L}$) while (\ref{SerreVanishing}) is satisfied.

Consider the section ring $R(\mathcal{X},\mathcal{L})=\bigoplus_{n \ge 0}H^0(\mathcal{X},\mathcal{L}^n)$, in which case $R(\mathcal{X},\mathcal{L})$ is shown to be normal by the condition that $\mathcal{X}$ is a normal scheme. 
By the K\"unneth formula (see \cite[Tag 0BED]{StacksProject}) together with $(\ref{SerreVanishing})$, $(\ref{Deligne2})$ and $H^1(\mathbf{P}^1_k,\mathcal{O}_{\mathbf{P}^1_k}(n))=0$ ($n\ge 1$), we have
$$
H^1\big(\mathcal{X} \times \Spec(k),\mathcal{L}^n \otimes_V k\big)=H^1(X,L^n)=\bigoplus_{i+j=1} H^i(Y,L_1^n) \otimes_k H^j(\mathbf{P}^1_k,L_2^n)=0,
$$
so that $(\ref{gradedisom})$ holds. We define $(R,\fm)$ to be the completion of $R(\mathcal{X},\mathcal{L})$ with respect to the adic topology defined by powers of the maximal ideal $(\pi)+R(\mathcal{X},\mathcal{L})_{+}$. The goal is to show that $(R,\fm)$ admits no small Cohen-Macaulay algebra.

Now assume that $R$ has a small Cohen-Macaulay algebra $T$. After tensoring $R \to T$ with $R/(\pi)$, we get a small Cohen-Macaulay $R/(\pi)$-algebra $T/(\pi)$. But 
then combining Lemma \ref{non-existence2} and $(\ref{gradedisom})$ yields a contradiction, because $R/(\pi)$ is the completion of $R(X,L)$ along the maximal ideal. We conclude that $R$ admits no small Cohen-Macaulay algebra.

\begin{claim}
$R(\mathcal{X},\mathcal{L})$ admits a small Cohen-Macaulay module. 
\end{claim}

\begin{proof}[Proof of claim]
Let us put
$$
M:=\bigoplus_{n \ge 0} H^0\big(\mathcal{X},\eta_1^*(\mathcal{L}_1)^{n+1}\otimes \eta_2^*(\mathcal{L}_2)^n\big),
$$
which is an $R(\mathcal{X},\mathcal{L})$-module by definition. Consider the commutative diagram
\begin{center}$\begin{CD}
X=Y\times_k\mathbf{P}^1_k @>\beta>> \mathcal{X}=\mathcal{Y}\times_V \mathbf{P}^1_V \\ @V \pi_1 VV @ V \eta_1 VV \\ Y @> \alpha_1 >> \mathcal{Y}
\end{CD}$
\end{center}
where $\beta$ and $\alpha_1$ are closed immersions with respect to the ideal sheaf $\widetilde{(\pi)}$. We have
$$
\alpha_1^*(\mathcal{L}_1)\cong L_1
$$
from the aforementioned lifting property. Letting $n\in \mathbb{N}_0$, since pullback commutes with tensor products, we have
$$
\beta^*\big(\eta_1^*(\mathcal{L}_1)^{n+1}\big)= (\eta_1 \circ \beta)^*(\mathcal{L}_1^{n+1})= (\alpha_1\circ \pi_1)^*(\mathcal{L}_1^{n+1})=\pi_1^*(L_1)^{n+1},
$$
which then yields
\begin{equation}
\label{EquationNumberONE}
\beta_*\big(\pi_1^*(L_1)^{n+1}\big)\cong \beta_*\Big( \beta^*\big(\eta_1^*(\mathcal{L}_1)^{n+1}\big)\Big)\overset{\text{\cite[Tag 04CI(1)]{StacksProject}}}{\cong}  \big(\eta_1^*(\mathcal{L}_1)^{n+1}\big)\otimes_V k.
\end{equation}

Arguing by considering a similar commutative diagram as above, we have $\beta^*\big(\eta_2^*(\mathcal{L}_2)^{n}\big)=\pi_2^*(L_2)^{n}$ and hence, 
\begin{equation}
\label{EquationNumberTWO}
\beta_*\big(\pi_2^*(L_2)^{n}\big)\cong \beta_*\Big(\beta^*\big(\eta_2^*(\mathcal{L}_2)^{n}\big)\Big)\overset{\text{\cite[Tag 04CI(1)]{StacksProject}}}{\cong} \big(\eta_2^*(\mathcal{L}_2)^{n}\big)\otimes_Vk
\end{equation}

Recalling that pushforward commutes with tensor product along closed immersions, $(\ref{EquationNumberONE})$ combined with $(\ref{EquationNumberTWO})$ implies that
$$
\forall\ n\in \mathbb{N}_0,~\beta_*\big(\pi_1^*(L_1)^{n+1}\otimes \pi_2^*(L_2)^{n}\big)=\big(\eta_1^*(\mathcal{L}_1)^{n+1}\otimes \eta_2^*(\mathcal{L}_2)^{n}\big)\otimes_Vk.
$$
Consequently, since $\beta:X\rightarrow \mathcal{X}$ is an affine morphism, we have 
\begin{align*} 
H^1\Big(\mathcal{X},\big(\eta_1^*(\mathcal{L}_1)^{n+1}\otimes \eta_2^*(\mathcal{L}_2)^{n}\big)\otimes_Vk\Big)&\cong H^1\big(X,\pi_1^*(L_1)^{n+1}\otimes \pi_2^*(L_2)^{n}\big)&\\&\overset{}{=}0
\end{align*}
by the K\"unneth formula \cite[Tag 0BED]{StacksProject}. So this implies that
$$
M/\pi M \cong \bigoplus_{n \ge 0} H^0\big(X,\pi_1^*(L_1)^{n+1}\otimes \pi_2^*(L_2)^{n}\big)\cong R(Y,L_1)(1)\# k[x,y].
$$
Therefore, $M$ is a small Cohen-Macaulay $R(\mathcal{X},\mathcal{L})$-module, as seen in the proof of Lemma \ref{non-existence2}(ii).
\end{proof}
So we end the proof of the theorem.
\end{proof}

\begin{remark}
\label{finalremark}
\begin{enumerate}
\item
The proof of Theorem \ref{non-existence1} contains the following principle: There is a good criterion to determine whether a given proper variety over a perfect field of characteristic $p>0$ admits a flat lifting to the ring of Witt vectors. It is also possible to lift a graded ring to the Witt vectors (simply lift the coefficients of the equations involved in the defining ideal). However, it is hard to check if the lifting can be taken to be flat.

\item
As already noted in \cite{SmallCM}, other projective varieties than $K3$ surfaces would suffice to achieve our goal, as long as the criterion, which is Theorem \ref{non-existence}, can be applied.
Let $X$ be a smooth projective variety over an algebraically closed field $k$ of characteristic $p>0$.
In \cite{WittVectorCoh}, it is shown that there is an isomorphism:
\begin{equation}
\label{WittVectorCohSeq}
H^i_{\rm{rig}}(X/W(k)[\frac{1}{p}])_{<1} \cong H^i(X,W\mathcal{O}_X)[\frac{1}{p}].
\end{equation}

Then the exact sequence of Abelian sheaves $0 \to \mathcal{O}_X \to W_{n+1}\mathcal{O}_X \to W_n\mathcal{O}_X \to 0$ in the Zariski topology gives a long exact sequence 
\begin{equation}
\label{WittLongExt}
\cdots \to H^i(X,\mathcal{O}_X) \to H^i(X,W_{n+1}\mathcal{O}_X) \xrightarrow{\delta_n} H^i(X,W_{n}\mathcal{O}_X) \to H^{i+1}(X,\mathcal{O}_X) \to \cdots,
\end{equation}
where $H^i(X,W_{n}\mathcal{O}_X)$ is the $n$-th truncated Witt-vector cohomology, due originally to Serre. Then we have an isomorphism $H^i(X,W\mathcal{O}_X) \cong \varprojlim_{n \to \infty} H^i(X,W_{n}\mathcal{O}_X)$ (c.f. \cite{WittVectorCoh}), where the transition map is given by $\delta_n$. Notice that $H^i(X,W_{0}\mathcal{O}_X)=H^i(X,\mathcal{O}_X)$ for any $i \ge 0$ by the formalism of Witt vectors. We obtain the following.
\begin{enumerate}
\item[$\bullet$] 
Under the notation as above, assume that $X$ satisfies the condition $H^i(X,\mathcal{O}_X) \ne 0$ and $H^{i+1}(X,\mathcal{O}_X)=0$ for some $0<i<\dim X$. Then $H^i(X,W\mathcal{O}_X) \ne 0$.
\end{enumerate}  
However, it is usually difficult to know whether $H^i(X,W\mathcal{O}_X)$ is $p$-torsion free, or not. Therefore, it is more desirable to explore the left side of the isomorphism $(\ref{WittVectorCohSeq})$
in order to find examples for our purpose, while the Witt-vector cohomology can be useful via $H^i(X,W_{0}\mathcal{O}_X)=H^i(X,\mathcal{O}_X)$ and $(\ref{WittLongExt})$. For recent results on $p$-torsion issues on various cohomology theories, we refer the reader to \cite{IntegralHodge}.
\end{enumerate}
\end{remark}

\subsection{An example arising from $p$-adic deformations of Abelian varieties}

 For the definition of a group scheme $G$ over a scheme $S$ (respectively, over a field $k$), see \cite[Tag 022S]{StacksProject}. For the definition and notable properties of their important subclass, Abelian varieties over a field, see \cite[Tag 03RO, Tag 0BFA and Tag 0BFC]{StacksProject}. 

For references to the results in the literature on Abelian varieties over a field, in parallel to \cite{Mumford}, we occasionally refer to \cite{EdixhovenGeerMoonenAbelian}, \cite{OortFinite} and \cite{StacksProject} because some  results are not explicitly written in \cite{Mumford} as we need them. Moreover, \cite{Mumford} often imposes a standing  assumption that the base field  of the considered Abelian varieties is algebraically closed which is not always needed.  The notation $X(k)$ for an Abelian variety $X$ over a field $k$, which is the \textit{set of $k$-valued points of $X$}, is the set of $k$-morphisms from $\Spec(k)$ to $X$ and it admits an abstract group structure induced by the scheme group structure of $X$.

We recall that an \textit{isogeny} of Abelian varieties (necessarily of the same dimension) is a surjective group morphism of Abelian varieties whose kernel is a finite group scheme (see e.g. \cite[page 63]{Mumford} and \cite[Definition (5.3)]{EdixhovenGeerMoonenAbelian}). The \textit{dual Abelian variety} $X^{\vee}$ of an Abelian variety $X$ over a field $k$ can be defined as the identity component $\Pic_{X/k}^0$ of the Picard scheme $\Pic_{X/k}$ (see \cite[Definition and Notation (6.19)]{EdixhovenGeerMoonenAbelian}, see also \cite[Chapter VI]{EdixhovenGeerMoonenAbelian} for the notion, some existence results and properties of the Picard scheme and its identity component; see also \cite[The identity component (3.16)]{EdixhovenGeerMoonenAbelian} for the definition of the \textit{identity component}). In particular, the abstract group of $k$-valued points of $X^{\vee}=\Pic^{0}_{X/k}$ is a subgroup of $\Pic(X)$, see \cite[page 125, second paragraph]{Mumford}. The \textit{dual Abelian variety} of $X$ is defined alternatively in \cite[III, \S 13]{Mumford} as a group quotient of $X$ (cf. \cite[Theorem (6.18)]{EdixhovenGeerMoonenAbelian}). Before proceeding, we recall the definition of polarized Abelian varieties.

\begin{remark}\label{RemarkMumfordBundlePolarization}
Let $L$ be a line bundle on an Abelian variety $X$ and let $\Lambda(L)$ be its associated  \textit{Mumford line bundle} on $X\times X$, i.e. $\Lambda(L)=m^*(L)\otimes \pr_1^*(L^{-1})\otimes \pr_2^*(L^{-1})$, where $m$ is the addition  morphism and $\pr_1$ (respectively, $\pr_2$) is the projection to the first (respectively, second) copy of  $X$ in $X\times X$. Using $\Lambda(L)$, one obtains a homomorphism $\varphi_L:X\rightarrow X^\vee$ (depending on the choice of  $L$) which at the level of $k$-points coincides with $a\in X(k)\mapsto t_a^*(L)\otimes L^{-1}$ (\cite[paragraph before  Theorem (6.18)]{EdixhovenGeerMoonenAbelian}, or \cite[paragraph before Lemma (2.3.1)]{OortFinite}). Here, $t_a$ is the \textit{translation morphism} $x \mapsto x+a$. More precisely, $t_a$ is the morphism
$$
\begin{CD}X@>\cong>>X\times k@>\id\times a>>X\times X@>\text{addition morphism}\ >>X\end{CD}.
$$
\end{remark}

\begin{definition} (\cite[Definition (11.6)]{EdixhovenGeerMoonenAbelian})
Let $X$ be an Abelian variety over a field $k$. A \textit{polarized} Abelian variety is a pair $(X,\lambda)$ where  $\lambda:X \to X^{\vee}$ is an isogeny from $X$ to its dual Abelian variety $X^\vee$ such that, up to a base change to a finite separable extension of $k$, $\lambda$ agrees with $\varphi_L$ for some ample line bundle $L$  (as described in Remark \ref{RemarkMumfordBundlePolarization}).
\end{definition}

This definition can be extended to an Abelian scheme.  
Moreover, by the basic theory of Abelian varieties, it is known that two ample line bundles $L$ and $M$ determine the same isogeny $\lambda:X \to X^\vee$ if and only if $t_a^*L \cong M$ for some $a \in X(k)$. So by abuse of notation, a specification of some ample line bundle $L$ on an Abelian variety $X$ will denote the polarization in what follows. The following theorem is of great value.

\begin{theorem}[Norman]
\label{NormanThm}
Let $k$ be either a perfect field of characteristic $p=3$, or an arbitrary field of characteristic $p\ge 5$. Assume that $(X,\lambda')$ is a polarized Abelian variety over $k$. Then there exists a polarized Abelian scheme $h:(\mathcal{X},\lambda) \to \Spec\big(W(k)[\sqrt{p}]\big)$ that specializes to $(X,\lambda')$ along the closed fiber of $h$.
\end{theorem}

This is the main theorem in Norman's paper \cite{Norman}, where he states the result in a more general form. In  Norman's theorem, an Abelian scheme over $\Spec\big(W(k)[\sqrt{p}]\big)$ is by definition a smooth and proper group scheme over $\Spec\big(W(k)[\sqrt{p}]\big)$ with geometrically connected fibers (whose group structure is 
necessarily commutative), see \cite[Definition 1.1]{FaltingsChaiDegeneration}. Moreover, see \cite[pages 2-4 and Definition 1.6]{FaltingsChaiDegeneration} for the definition of the dual Abelian scheme as well as a polarization  of the Abelian scheme $\mathcal{X}$  in the statement of  Theorem \ref{NormanThm} (see also \cite[page 233 and Theorem 5]{BoschNeron} for more explanation concerning the dual Abelian schemes). These are defined using an open subscheme of the  Picard scheme generalizing  the case of Abelian varieties over a field as mentioned above.




In order to lift an ample line bundle on an Abelian variety,  one can alternatively lift its associated polarization, which is itself of independent interest:

\begin{lemma}
\label{LemmaLiftingPolizarationAndLiftingOfLineBundles}
Let $R$ be a complete discrete valuation domain with residue field $k$. Let $(X,\lambda)$ (respectively, $(X',\lambda')$) be a polarized Abelian scheme (respectively,  polarized Abelian variety) over $\Spec(R)$ (respectively,  $\Spec(k)$) such that $\lambda$ lifts $\lambda'$ and $\lambda'=\varphi_{L'}$ for some ample line bundle $L'$  over $X'$. Then there exists an ample line bundle $L$ on $X$ lifting $L'$.	
\end{lemma}

\begin{proof}
Let $y$ be the closed point of $\Spec(R)$ defining $\mathfrak{G}:=\text{Spf}(R)$.  Also  let   $\mathfrak{X}$  be the formal completion of $X$  along the   inverse image of $y$ in $X$. We notice that the  morphism of formal schemes $\mathfrak{X}\rightarrow \mathfrak{G}$, induced by the proper morphism $X\rightarrow \Spec(R)$, is proper (see \cite[(3.4.1), page 119]{GrothendieckEGAIII} for the definition of a proper morphism of formal schemes). Also $R$ is an adic ring in the sense of \cite[Definition (7.1.9)]{GrothendieckEGAI} with the principal prime ideal of definition $\fm=\pi R$. Thus we can appeal to \cite[Theorem 5.4.5]{GrothendieckEGAIII},  which shows that the statement follows from the existence of a line bundle on $\mathfrak{X}$ lifting $L'$.

For each $n\in \mathbb{N}$, we set  $\mathcal{X}_n=\mathcal{X}\times_{\mathfrak{G}} (\Spec(R/\pi^n R))(=X_n:=X\times_{\Spec(R)}(\text{Spec}(R/\pi^n R)))$. We recall that $X^\vee\times \Spec(R/\pi^n R)$ coincides with $\big(X\times \Spec(R/\pi^n R)\big)^\vee$. Thus, according to our hypothesis, by applying inductively \cite[Lemma (2.3.2)]{OortFinite} to the small surjection $R/\pi^2 R\rightarrow k$ and then to the small surjection $R/\pi^3 R\rightarrow R/\pi^2 R$ and so on, we obtain a sequence of ample line bundles $L_n$ over $X_n=\mathfrak{X}_n$ such that $L_{n+1}$ lifts $L_n$  and $\varphi_{L_{n+1}}$ coincides with $\lambda\times \Spec(R/\pi^{n+1} R)$ for each $n\in \mathbb{N}$ (for this conclusion one may need to use \cite[Lemma (2.3.1), diagram (2)]{OortFinite} besides \cite[Lemma (2.3.2)]{OortFinite}).

To complete the proof, it remains only to show that the Mittag-Leffler condition in \cite[II, Exercise 9.6]{HartshorneAlgebraic} is available in our situation: because by applying  \cite[II, Exercise 9.6]{HartshorneAlgebraic} to the sequence $(L_n)_{n\in \mathbb{N}}$ of invertible sheaves, we then obtain an invertible sheaf $L$ on $\mathfrak{X}$  lifting $L'$, as was to be found (see also \cite[II, Proposition 9.6(b)]{HartshorneAlgebraic}).  Finally, setting $\mathscr{G}:=\pi\mathcal{O}_X$, the Mittag-Leffler condition for the inverse system  $$\big(\Gamma(\mathfrak{X}_{n},\mathcal{O}_{\mathfrak{X}_{n}})\big)_{n\in \mathbb{N}}=\big(\Gamma(X_n,\mathcal{O}_{X_{n}})\big)_{n\in\mathbb{N}}=\big(\Gamma(X,\mathcal{O}_{X}/\mathscr{G}^n)\big)_{n\in\mathbb{N}}$$ holds in light of \cite[Tag 02OB]{StacksProject}.
\end{proof}

The next lemma is the last preparatory fact  we need for the proof of Theorem \ref{TheoremAbelianSurface}.

\begin{lemma}\label{LemmaVanishingOfHigherCohomologies}
Let $X$ be an Abelian variety over a field $k$ and $L$ be an ample line bundle on $X$ with a non-zero global section. Then $H^i(X,L)=0$ for any $i\neq 0$.
\end{lemma}

\begin{proof}
Let $\overline{k}$ be the algebraic closure of $k$ and set $X_{\overline{k}}:=X\times_k\overline{k}$ which is an Abelian variety over $\overline{k}$. Let $f:X_{\overline{k}}\rightarrow X$ be the projection and set $L_{\overline{k}}:=f^*(L)$ which is a line bundle over $X_{\overline{k}}$. 
The projection $X_{\overline{k}}\rightarrow X$ is obviously an affine morphism and consequently $L_{\overline{k}}$ is also ample in view of \cite[Tag 0892]{StacksProject}.


Let $K(L_{\overline{k}})=\{x\in X_{\overline{k}}(\overline{k}):t_{x}^*(L_{\overline{k}})\cong L_{\overline{k}}\}$ be the subgroup of $X_{\overline{k}}(\overline{k})$ as defined in \cite[Definition, page 60]{Mumford}. 
By our hypothesis, $L$ and thus $L_{\overline{k}}$ has a non-zero global section.  In particular, $L_{\overline{k}}=\mathcal{O}_{X_{\overline{k}}}(D)$ for some effective Cartier divisor $D$ on $X_{\overline{k}}$ (\cite[II, Proposition 7.7(a)]{HartshorneAlgebraic}). Therefore, \cite[Application 1, page 60]{Mumford} implies that $K(L_{\overline{k}})$ is finite. This fact as well as \cite[The vanishing theorem, page 150]{Mumford} implies that $H^i(X_{\overline{k}},L_{\overline{k}})=0$ for $i\neq 0$, because $H^0(X_{\overline{k}},L_{\overline{k}})\neq 0$. Then \cite[Tag 02KH(2)]{StacksProject} implies the statement.
\end{proof}

Now we are ready to present the main results of this section.

\begin{theorem}
\label{TheoremAbelianSurface}
Let $X$ be an Abelian variety of dimension $d\ge 2$ over a field $k$ of arbitrary characteristic and let
$L$ be   any ample line bundle on $X$.
\begin{enumerate}[(i)]


\item
\label{itm:SectionRingOFAbelianSurface}
Let $D$ be a $\mathbb{Q}$-Weil divisor on $X$ such that $ND$ is integral and $\mathcal{O}_X(ND)=L$ for some $N\ge 1$. The non-Cohen-Macaulay normal $(d+1)$-dimensional generalized section ring
$$
R(X,D):=\bigoplus_{i\ge 0}H^0\big(X,\mathcal{O}_X(\lfloor iD\rfloor)\big)
$$
admits a graded small Cohen-Macaulay module (so does the section ring $R(X,L)$ itself). 

\item 
\label{itm:AbelianNonAdmittingSCMalgebra}
The completion $\widehat{R(X,L)}$ of $R(X,L)$ along the maximal ideal $R(X,L)_+$ does not admit a small Cohen-Macaulay algebra.

\item 
\label{itm:AbelianMixedCharacteristic}
Further assume that either $k$ has prime characteristic	 $p \ge 5$, or it is a perfect field of characteristic $p=3$.
Then, up to replacing $L$ with some sufficiently large power  $L^t$ of it $(t\in \mathbb{N})$, there exists a $(d+2)$-dimensional complete normal local domain $(R,\fm)$, which is flat over the discrete valuation ring $W(k)[\sqrt{p}]$, such that $R/\sqrt{p}R \cong \widehat{R(X,L)}$. In particular, $R$ does not admit any small Cohen-Macaulay algebra, but it admits a small Cohen-Macaulay module.
\end{enumerate}
\end{theorem}

\begin{proof}
\ref{itm:SectionRingOFAbelianSurface}
First we note that as $H^1(X,\mathcal{O}_X)\neq 0$ by \cite[Corollary 2, page 129]{Mumford} and \cite[Tag 02KH(2)]{StacksProject}, $R(X,D)$ is not Cohen-Macaulay (see \cite[Propostion 2.2]{WatanabeSomeRemarks}). 

To observe that $R(X,D)$ admits a maximal Cohen-Macaulay module, we may and we do choose some $k_0\in \mathbb{N}$ such that $L^{k}$ and $(-1_X)^*L^k$ are very ample for $k\ge k_0$ (we recall that the inverse morphism $-1_X:X\rightarrow X$ is an automorphism).\footnote{On Abelian varieties over an algebraically closed field, such a power is not that large (\cite[Theorem 163]{Mumford}).}  Then we fix some arbitrary $N'\ge k_0$. Set $M:=2NN'$ and  let $M_X:X\rightarrow X$ be the endomorphism on $X$ corresponding to the multiplication by $M$, which is a finite locally free morphism by \cite[Tag 0BFG]{StacksProject}.  Before going on, we should  contemplate  $M^*_XD$.

As $X$ is nonsingular (\cite[Tag 0BFC and Tag 056S]{StacksProject}), every Weil divisor on $X$ is Cartier and every rank $1$ coherent reflexive sheaf of $\mathcal{O}_X$-modules	 is invertible. Therefore, we	  have a finite sum presentation   $$
D=\sum\limits_{i}p_i/q_iD_i
$$
of $D$, where each $D_i$ is a Cartier divisor on $X$, and $p_i,q_i$ are non-zero coprime integers with $q_i\ge 1$.

For the notion of the pullback of a Cartier divisor along a morphism, we refer to \cite[Definition (21.4.2)]{GrothendieckEGAIVIV} (whose existence is conditional in contrast to the pullback of general invertible sheaves). Appealing to the flatness of $M_X$ as well as  \cite[Proposition (21.4.5)]{GrothendieckEGAIVIV}, every Cartier divisor on $X$ has a pullback along $M_X$. In other words, we are provided with the endomorphism  $M_X^*:\text{Div}(X)\rightarrow \text{Div}(X)$ on the ordered  group of the Cartier divisors on $X$, which extends naturally to an endomorphism on the group of rational Cartier divisors $$M^*_X:\text{Div}(X)\otimes_{\mathbb{Z}}\mathbb{Q}\rightarrow \text{Div}(X)\otimes_{\mathbb{Z}}\mathbb{Q}.$$ Moreover, by \cite[(21.4.2.1), page 266]{GrothendieckEGAIVIV}, this group homomorphism $M^*_X$ is so that the diagram 
\begin{center}$\begin{CD}
\text{Div}(X) @>>> \Pic(X)\\
@VM^*_X VV @VV M^*_X V\\
\text{Div}(X) @>>> \Pic(X),
\end{CD}$\end{center}
is commutative, where its horizontal maps are the canonical group homomorphism  Div($X$)$\rightarrow $Pic($X$), inducing the canonical group isomorphism $\text{Div}(X)/\text{Div.princ}(X)\overset{\cong}{\rightarrow}\text{Pic}(X)$ (\cite[Proposition (21.3.4)(b)]{GrothendieckEGAIVIV}).  From \cite[Corollary 3, page 59]{Mumford} or \cite[Tag 0BFF]{StacksProject} and the above diagram, we get 
$$
M_X^*D\overset{(\text{Div.princ}(X))\otimes_{\mathbb{Z}}\mathbb{Q}}{\equiv} (M(M+1)/2)D+(M(M-1)/2)(-1_X)^*(D),
$$ 
\begin{equation}
\label{EquationPullbackOfCartierDivisor}
\mathcal{O}_X(M^*_XD)\cong L^{N'(M+1)}\otimes (-1_X)^*(L)^{N'(M-1)}.
\end{equation} 

In particular, $M_X^*D$ is an integral divisor and it is very ample by our choice of $M$ and \cite[II, Exercise 7.5(d)]{HartshorneAlgebraic}. So the section ring $R(X,M^*_XD)$ is a  Noetherian ring. Also, it is not hard to verify that the finiteness of $M_X:X\rightarrow X$ implies that the ring map on the generalized section rings $$R(X,D)\rightarrow R(X,M_X^*D),$$ induced by $M_X$, is module-finite, as the Proj of any generalized section ring of any ample $\mathbb{Q}$-Weil divisor on $X$ coincides with $X$; for the case of ample (integral) divisors, one can alternatively deduce this from \cite[Lemma 1.3]{MoriGradedFactorial} (which is a more general result as it has no ample/proper assumption) in conjunction with \cite[Tag 02JJ]{StacksProject}. 

Therefore, it suffices to show that $R(X,M_X^*D)$ admits a maximal Cohen-Macaulay module. We show that the $R(X,M_X^*D)$-module
$$
R(X,L^{k_0},M_X^*D):=\bigoplus_{n\in \mathbb{Z}} H^0\big(X,L^{k_0}\otimes \mathcal{O}_X(M_X^*D)^n\big)
$$
is maximal Cohen-Macaulay, or equivalently $H^i\big(X,L^{k_0}\otimes \mathcal{O}_X(M_X^*D)^n\big)=0$ for each $n\in \mathbb{Z}$ and $1\le i<d$ (see \cite[Proposition 2.1.5]{GrothendieckEGAIII} for the equivalence). Using \cite[II, Exercise 7.5(d)]{HartshorneAlgebraic} and by 
our choice of $N',k_0$  as well as (\ref{EquationPullbackOfCartierDivisor}),  $$L'_n:=L^{k_0}\otimes \mathcal{O}_X(M_X^*D)^n$$ 
is evidently very ample  for any $n\ge 0$.
So from Lemma \ref{LemmaVanishingOfHigherCohomologies}, we get $H^i(X,L'_n)=0$ for $n\ge 0$ and $1\le i<d$.
Also for the case where $n<0$, using  Serre duality theorem together with the fact that the canonical divisor of $X$ is the trivial line bundle (because  K\"ahler differentials of $X$ is the trivial vector bundle by \cite[Tag 047I]{StacksProject}), we have
$$
H^i(X,L'_n)\cong 
H^{d-i}\Big(X,L^{(-n)N'(M+1)-k_0}\otimes (-1_X)^*(L)^{(-n)N'(M-1)}\Big)
$$
which vanishes for all $1\le i<d$ by exactly the same reason as above. So $R(X,L^{k_0},M_X^*D)$ provides us with a maximal Cohen-Macaulay $R(X,D)$-module, as was to be proved.

\ref{itm:AbelianNonAdmittingSCMalgebra}
If $k$ has characteristic zero, then the statement is well-known in view of the normality of section rings, using the normalized trace map.  Moreover  by a base change to the algebraic closure of $k$ if necessary, without loss of generality  we can assume that $k$ is a perfect field. So we assume that $k$ is a perfect field of  prime characteristic.

 To verify that the completion $\widehat{R(X,L)}$ of $R(X,L)$ does not admit a small Cohen-Macaulay algebra, it amounts to checking the assumption of Theorem \ref{non-existence}. Namely, $H^1_{\rm{rig}}(X/W(k)[\frac{1}{p}])_{<1} \ne 0$. A proof for this needed non-vanishing is outlined in \cite[Remark 3.2]{SmallCM}.\footnote{See Remark \ref{AlbaneseVar}.} For the case of Abelian surfaces, however, another reference exists: the required non-vanishing in dimension $2$ can be deduced by measuring the slope of the line segments in each of the three possible  Newton polygons (based on the $p$-rank of $X$) of the first crystalline cohomology that can occur for an Abelian surface $X$, as drawn in \cite[page 651]{IllusieComplexDe}; see also \cite[3.2. Newton and Hodge polygons]{LiedtkeLectures} for the definition of the Newton polygon.

\ref{itm:AbelianMixedCharacteristic} 
In view of the definition of the polarization, we can consider the pair $(X,\varphi_L)$ as a polarized Abelian variety. However, to fulfill the required assumption $H^1(X,L^n)=0$ for $n\ge 1$ as in the statement of Proposition \ref{crucial-isom}, we have to (and we do) replace $L$ with $L^t$ for some $t\gg 0$.

By Theorem \ref{NormanThm} and Lemma \ref{LemmaLiftingPolizarationAndLiftingOfLineBundles}, there exists a polarized Abelian scheme $h:(\mathcal{X},\lambda) \to \Spec\big(W(k)[\sqrt{p}]\big)$ such that $(\mathcal{X},\lambda)$ specializes to $(X,\varphi_{L})$ along the special fiber of $h$ as well as  a line bundle $\mathcal{L}$ on $\mathcal{X}$  lifting $L$. Now the rest of the proof can be completed by applying 
Proposition \ref{crucial-isom}. Note that for sufficiently large $M$, following the arguments given in the proof of part \ref{itm:SectionRingOFAbelianSurface}, $R(\mathcal{X},\mathcal{L})\rightarrow R(\mathcal{X},M_{\mathcal{X}}^*\mathcal{L})$ is a finite ring map (which specializes to the finite ring map $R(X,L)\rightarrow R(X,M_{X}^*L)$) and the $R(\mathcal{X},M_{\mathcal{X}}^*\mathcal{L})$-module $R(\mathcal{X},\mathcal{L}^{k_0},M_{\mathcal{X}}^*\mathcal{L})$, which lifts the maximal Cohen-Macaulay $R(X,L)$-module $R(X,L^{k_0},M_{X}^*L)$, provides us with a maximal Cohen-Macaulay $R(\mathcal{X},\mathcal{L})$-module.
\end{proof}

\begin{remark}
\label{AlbaneseVar}
\begin{enumerate}[(i)]
\item\label{itm:AlbaneseVarNonExistenceOfSMC}
Let $X$ be a smooth proper variety over an algebraically closed field $k$ of characteristic $p>0$. Then one can attach an Abelian variety $\Alb(X)$ (called the Albanese variety) and a morphism $\alpha:X \to \Alb(X)$ (the algebraically closedness of $k$ is imposed here to ensure us that $X$ has a $k$-rational point or a $0$-cycle of degree $1$). Then by \cite[p. 171]{Enriques1} or \cite[Example 1.8(2)]{LiedtkeLectures}, we have
$$
H^1_{\rm{cris}}(X/W(k)) \cong H^1_{\rm{cris}}(\Alb(X)/W(k)), 
$$
which, in particular, says that the first crystalline cohomology is a torsion free $W(k)$-module. If the first \'etale Betti number of $X$ vanishes, then $\dim \Alb(X)=0$ and $H^1_{\rm{rig}}(X/W(k))$ vanishes (see \cite[Theorem 0.9.17]{Enriques1}). Using these facts as well as Lemma \ref{LemmaBlowUpAndSegreProduct},
we can provide an alternative proof of a result in \cite[Example 5.2 and Example 5.3]{SannaiSinghGalois}. Indeed, let $X:=E \times \mathbf{P}_k^1$, where $E$ is the Fermat cubic curve, where the characteristic of $k$ is different from $3$. Then $X$ is a smooth projective surface and we have $\Alb(X) \cong \Alb(E) \times \Alb(\mathbf{P}_k^1)$ (see \cite[p. 10-05, p. 10-10]{Serre} for the decomposition and the universality of Albanese varieties). Since $\Alb(\mathbf{P}_k^1)$ is a point, we get $H^1_{\rm{cris}}(X/W(k)) \cong H^1_{\rm{cris}}(\Alb(E)/W(k)) \cong H^1_{\rm{cris}}(E/W(k))$. This is a torsion free $W(k)$-module of rank $2$. By the possibilities of the Newton/Hodge polygon (see \cite[3.4.1]{LiedtkeLectures}), or by \cite[Remark 3.2]{SmallCM}, we deduce that the slope $<1$ part of 
\begin{align*}
	H^1_{\text{rig}}(X/W(k)[1/p])&\cong H^1_{\text{cris}}(X/W(k))\otimes_{W(k)} W(k)[1/p]&\\&\cong H^1_{\text{cris}}(E/W(k))\otimes_{W(k)}W(k)[1/p]
\end{align*}
is nonzero, where the rigid cohomology is identified with crystalline cohomology after tensoring with the field of fractions of $W(k)$. Thus, Theorem \ref{non-existence} shows that the homogeneous coordinate ring of $X$ (this is the same thing as the Segre product: $k[X,Y,Z]/(X^3+Y^3+Z^3) \# k[s,t]$) does not admit a graded module-finite extension that is maximal Cohen-Macaulay.

\item
Recently, it is established that any globally $F$-split normal projective surface admits a flat lifting over the Witt vectors in \cite{LiftingF-split}. This is a singular example admitting a $p$-adic deformation. We are interested in testing this example to provide a new example of a complete local domain without small Cohen-Macaulay algebras.
\end{enumerate}
\end{remark}

The final result of this paper is an immediate corollary of Theorem \ref{TheoremAbelianSurface}.

\begin{corollary}\label{CorollaryAbelian}
Let $R$ be a normal  Noetherian  $\mathbb{N}_0$-graded ring  over a field $k=R_{[0]}$. If $\Proj(R)$ is an Abelian variety over $k$, then $R$ admits a graded small Cohen-Macaulay module.
\end{corollary}

\begin{proof}
By virtue of \cite[Theorem, page 51]{DemazureAnneaux}, there exists an ample $\mathbb{Q}$-Weil divisor $D$ on $X=\Proj(R)$ such that $R=R(X,D)$ (see also \cite[Remark 2.6(i)]{ShimomotoTavanfarRemarks}). Hence the statement follows from Theorem \ref{TheoremAbelianSurface}\ref{itm:SectionRingOFAbelianSurface}.
\end{proof}

In Theorem \ref{TheoremAbelianSurface}, using an Abelian surface, we have constructed a complete local normal domain $(R,\fm)$ of positive characteristic $p>0$ in dimension 3 such that $R$ does not admit a small Cohen-Macaulay algebra. However, since every complete local normal domain in dimension 2 is automatically Cohen-Macaulay, our method using the $p$-adic deformation is not sufficient to construct a complete local normal domain of mixed characteristic in dimension 3 that has no small Cohen-Macaulay algebras. We propose the next problem.

\begin{Problem}
Does every complete local normal domain of mixed characteristic in dimension 3 admit a small Cohen-Macaulay algebra?
\end{Problem}

\section{Acknowledgment}
The authors are grateful to an anonymous reader who pointed out to us Example \ref{ExamplePrimeNonExistenceExample}. The authors are also grateful to Bhargav Bhatt for an enlightening comment on Remark \ref{AlbaneseVar}\ref{itm:AlbaneseVarNonExistenceOfSMC}.

\end{document}